\documentclass[draft,11pt]{amsart}
\usepackage{amssymb}

\theoremstyle{plain}

\theoremstyle{plain}
\newtheorem{thm}{Theorem}[section]

\newtheorem{cor}[thm]{Corollary}
\newtheorem{prop}[thm]{Proposition}
\newtheorem{lem}[thm]{Lemma}
\newtheorem{defn}[thm]{Definition}
\newtheorem{rem}[thm]{Remark}
\newtheorem{notation}[thm]{Notation}
\newtheorem{remark}[thm]{Remark}
\newtheorem{prob}{Problem}
\newtheorem{sub}[thm]{Sublemma}

\newcommand{\e}{\varepsilon}

\newcommand{\R}{\mathcal R}

\newcommand{\supp}{{\rm supp}\,}
\newcommand{\maxsupp}{{\rm maxsupp}}
\newcommand{\minsupp}{{\rm minsupp}}

\newcommand{\ran}{{\rm ran}\,}

\newcommand{\N}{\mathbb{N}}

\newcommand{\xfr}{\mathfrak{X}}

\newcommand{\ssh}{\mathcal{S}}
\newcommand{\ssm}{\mathcal{S}^M}

\def\e{\varepsilon}
\def\N{\mathbb N}

\def\L{\mathcal L}
\def\A{\mathcal A}

\def\n{\noindent}
\def\a{\alpha}
\def\w{\omega}
\def\b{\beta}
\def\g{\gamma}
\def\l{\lambda}

\begin{document}
\baselineskip 18pt

\title[A weak Hilbert space with few symmetries]{A weak Hilbert space with few symmetries}

\author[Spiros A. Argyros]{Spiros A. Argyros}
\address{Department of Mathematics,
    Zografou Campus,
    National Technical University,
    Athens 157 80, Greece}
\email{sargyros@math.ntua.gr}

\author[Kevin~Beanland]{Kevin~Beanland}
\address{Department of Mathematics and Applied Mathematics,
    Virginia Commonwealth University,
    Richmond, VA 23284}
\email{kbeanland@vcu.edu}

\author[Theocharis Raikoftsalis]{Theocharis Raikoftsalis}
\address{Department of Mathematics,
    Zografou Campus,
    National Technical University,
    Athens 157 80, Greece}
\email{th-raik@hotmail.com}



\date{}

\begin{abstract}
We construct a weak Hilbert Banach space such that for every block subspace $Y$ every
bounded linear operator on $Y$ is of the form $D+S$ where $S$ is a strictly singular operator and $D$ is a diagonal operator.  We show that this yields a weak Hilbert space whose block subspaces are not isomorphic to any of their proper subspaces.
\end{abstract}

\maketitle

\section{Introduction}

The weak Hilbert spaces form a class of Banach spaces including
Hilbert spaces that share many of their important properties. We
recall their definition. An infinite dimensional Banach space $X$
is called a weak Hilbert space if there exist positive numbers
$\delta, C$ such that every finite dimensional space $E \subset X$
contains a subspace $F\subset E$ such that $\dim F \geq \delta
\dim E$, the Banach-Mazur distance between $F$ and $\ell_2^{\dim
F}$ is at most equal to $C$ and there is a projection $P:X
\rightarrow F$ with $\|P\|\leq C$, ($\ell_2^n$ denotes the Hilbert
space of dimension $n$). The above definition finds its origins in
the seminal works of V. Milman and G. Pisier \cite{MP} and G.
Pisier \cite{Pi1}, however, the book of G. Pisier \cite{Pi2}
remains the most comprehensive reference for weak Hilbert spaces.
In \cite{Pi2} one can find numerous characterizations and an
in-depth discussion of the properties of these spaces.  All
subspaces, quotients and duals of weak Hilbert spaces are
themselves weak Hilbert.  The Fredholm theory as developed by
Grothendieck \cite{Gr} works in weak Hilbert spaces as well as
Hilbert spaces and W.B. Johnson (unpublished, see \cite{Pi2})
showed that all weak Hilbert spaces are superreflexive.

When one considers the rich structure and geometry of Hilbert
space, it is natural to ask what kind of geometry a weak Hilbert
space must possess. In particular, it is very interesting to
investigate how divergent can the global geometry of such a space
be when compared with the local Hilbertian structure. The most
significant step in this direction, was made by W.B. Johnson in
\cite{Jo2} where it was shown that the 2-convexification of the
modified Tsirelson space is a weak Hilbert space with {\em no}
subspace isomorphic to a Hilbert space. Our aim in the present
paper is to construct a weak Hilbert space having a quite
divergent structure from that of a Hilbert space. Namely, we
construct a Banach space $\xfr_{wh}$ with an unconditional basis
$(e_n)_n$ that has the following properties:

\begin{enumerate}
\item $\xfr_{wh}$ is a weak Hilbert space. \item For every block
subspace $Y$ of $\xfr_{wh}$ every operator in $\L(Y)$ takes the
form $D|_{Y}+S$ where $De_n=\lambda_n e_n$ for some scalar
sequence $(\lambda_n)$ and $S\in \L (Y)$ is strictly singular.
\item Every block subspace $Y$ of $\xfr_{wh}$ is not isomorphic to
any of its proper subspaces. \item The space $\xfr_{wh}$ does not
contain a quasi minimal subspace.
\end{enumerate}
 In the above(and herein) we use $\L(Y)$ to denote the Banach space of bounded
linear operators on $Y$. A operator $S\in \L(Y)$ is {\em strictly
singular} if its restriction to any infinite dimensional subspace
is not an isomorphism.  In the sequel, we call an operator $D$
{\em diagonal} if $De_n=\lambda e_n$ for some scalar sequence
$(\lambda_n)$ and some {\em a priori} fixed basis $(e_n)$.
Although the space we construct is built over the field of real
numbers, by applying exactly the same methods one can naturally
extend the construction to that of space with the same properties
defined over the field of complex numbers. In both cases, we
correlate the spectrum of an operator $T\in \L (Y)$ with that of
the diagonal operator $D$, where $T=D+S$. We also recall that a
space $X$ is called quasi minimal if it does not contain a pair of
totally incomparable subspaces.

 There are several criteria for showing that a space is
weak Hilbert. One of them concerns spaces with a Schauder basis
and relates to how the norm behaves on disjointly supported
vectors. More precisely, N. J. Nielsen and N. Tomczak-Jaegermann
in \cite{NT}, by applying theorems of W.B. Johnson \cite{Jo1},
show that a space with a basis is weak Hilbert if the basis is
asymptotic $\ell_2$ for vectors with disjoint support.  A good
reference for this proof is \cite{ACK}.  We recall the definition
of this notion here. A space $X$ with a basis $(e_i)$ is a {\em
asymptotic $\ell_2$ for vectors with disjoint supports} if there
is a $C \geq 1$ such that for every $n \in \N$, every sequence of
disjointly supported vectors $(x_i)_{i=1}^n$ with $n \leq \supp
x_i$ for $i \leq n$, $(x_i)_{i=1}^n$ is $C$-equivalent to the unit
vector basis of $\ell_2$. Using this condition, Edgington
\cite{E}, and then later Androulakis, Casazza, and Kutzarova
\cite{ACK}, constructed non-trivial weak Hilbert spaces each with
an unconditional basis and saturated with copies of $\ell_2$. The
definition of the space $\xfr_{wh}$ presented in this paper
utilizes a type of modified mixed Tsirelson saturation method
which yields the aforementioned property. In the next section we
give a description of the norm of $\xfr_{wh}$ and further discuss
some of its critical properties.

\section{Description of $\xfr_{wh}$}

Let $c_{00}$ denote the vector space of the finitely supported
scalar sequences and $(e_n)$ denote the unit vector basis of
$c_{00}$.  For $x=\sum_{i=1}^\infty a_i e_i \in c_{00}$, let
$\supp x = \{ i \in \N : a_i \not= 0 \}$ and $\ran x$ be the
smallest interval containing $\supp x$.  For $E,F \subset \N$,
write $E < F$ if $\max E < \min F$ or either $E$ or $F$ is empty.
For $x, y \in c_{00}$ we write $x<y$ and call $x,y$ {\em
successive} if $\ran x < \ran y$.  We write $n < x$ if $\{n\} <
x$.  For $E \subset \N$ and $ \sum_{i =1}^\infty a_i x_i \in
c_{00}$ let $Ex = \sum_{i \in E} a_i x_i$.

The definition of the space $\xfr_{wh}$ uses an injective function
with range a subset of the natural numbers or what has been
referred to as a coding function.  Codings functions were first
introduced by B. Maurey and H.P. Rosenthal in \cite{MR} where they
construct a weakly null sequence with no unconditional
subsequence.  They have become a ubiquitous component of
constructions of spaces with few operators and, as in our case,
few symmetries.  In \cite{Go2}, W.T. Gowers constructs the first
example of a space $\xfr_{gu}$ not isomorphic to any of its
hyperplanes.  The operators on this space take the form $D+S$
where $D$ is diagonal and $S$ is strictly singular (c.f.
\cite{GM2}). This is the first example of a space whose
construction uses a coding function but which had an unconditional
basis; our construction is similar in this way.  On the other
hand, an important new feature of our construction  is that it
admits an implicit description that is similar to the modified
mixed Tsirelson spaces described in \cite{ACK,ADKM,ADM,E}.

To help the reader better understand what is meant by this, we
state the implicit equations the norm satisfies. For $n \in \N$,
let $\ssh^M_n$ denote the modified Schreier family of order $n$
(see Section 3 for definitions). A finite family $(E_i)_{i=1}^d$
of pairwise disjoint finite subsets of $\N$ is called $\ssh_n^M$
allowable if $\{\min E_i\}_{i=1}^d\in\ssh_n^M$. The norm requires
two increasing sequences $(m_j)_{j=0}^\infty$ and
$(n_i)_{i=0}^\infty$ satisfying certain growth conditions and so
let us fix these throughout.  The norm of $\xfr_{wh}$ is the
completion of $c_{00}$ in the norm $\|x\|=\max\{\sup \{ \|x \|_j :
j \in \N\cup\{0\}\},\|x\|_{\infty}\}$ where the norms
$\|\cdot\|_j$ satisfy the following implicit formulas:

\begin{equation*}
\begin{split}
& \| x \|_{2j} =\sup \bigg\{ \frac{1}{m_{2j}} \bigg( \sum_{i=1}^k \|E_i x \|^2 \bigg)^\frac{1}{2} : (E_i)_{i=1}^k \mbox{ is $S^M_{n_{2j}}$ allowable }\bigg\}, \\
& \|x\|_{2j+1}=\sup \bigg\{ \frac{1}{m_{2j+1}}  \bigg(
\sum_{i=1}^k \|E_i x \|_{2j_i}^2 \bigg)^\frac{1}{2} : (E_i,
2j_i)_{i=1}^k \mbox{ is a $S^M_{n_{2j+1}}$-$\sigma$ special
sequence}\bigg\}.
\end{split}
\end{equation*}

\n The $\|\cdot\|_{2j+1}$ norms and special sequences are the key
ingredients in showing that our space has the asymmetry desired.
The special sequences impose on the space the non-homogeneous
structure.  We briefly outline some properties of special
sequences; the exact definition can be found in Section 3. Readers
familiar with previous constructions will notice many similarities
(such as the `tree-like-property').  Let $N_1$ and $N_2$ be
infinite subset of $\N$ with $\N=N_1\cup N_2$. Let
\begin{equation*}
\begin{split}
\Sigma = \{(E_i,2j_i)_{i=1}^n: &~ E_i \cap E_j = \emptyset \mbox{ and } j_1 < j_2 < \cdots < j_n \\
& \text{ with $j_1 \in N_1$ and $j_i \in N_2$ for $i > 1$}\}.
\end{split}
\end{equation*}

\n Special sequences $(E_i, 2j_i)_{i=1}^n \in \Sigma$ have two
important properties:
\begin{itemize}
\item[(i)] (Extension Property) Suppose $(E_i, 2j_i)_{i=1}^n$ is a
special sequence and $E \subset \N$ finite with $E \cap
(\cup_{i=1}^n E_i )=\emptyset$.  Then there exist an $j \in \N$
such that $((E_i,2j_1), \ldots , (E_n,2j_n),(E,2j))$ is also a
special sequence. \item[(ii)] (Tree-like property) If $(E_i,
2j_i)_{i=1}^n$ and $(F_l, 2k_l)_{l=1}^m$ are both special
sequences then either $j_i \not= k_l$ for all $i=1,\ldots, n$ and
$l=1, \ldots, m$ or there exists a $d \leq  \min\{n,m\}$ such that
$(E_i,2j_i)=(F_i,2k_i)$ for all $i \leq d-1$, $j_d=k_d$, but $E_d
\not=F_d$ and $j_i \not=k_i$ for all $d+1 \leq i \leq
\min\{n,m\}$.
\end{itemize}

It is not difficult to see that for each $j \in \N$, each $\|\cdot
\|_j$ is equivalent to the original norm $\| \cdot \|$.  The
behavior of the odd indexed norm is the most critical in
prescribing the asymmetric properties of the space.  In
particular, the norms $\|\cdot\|_{2j+1}$ exhibit the following
seemingly contradictive behavior.  In every block subspace, on one
hand, we can find a normalized sequence $(x_k)_{k=1}^d$, a special
sequence $(E_k,2j_k)_{k=1}^d$ and a $(b_k)_{k=1}^d\in Ba(\ell_2)$
satisfying: $(\maxsupp x_k)_{k=1}^d, (\min E_k)_{k=1}^d \in
S_{n_{2j+1}}$,  $\|x_k\|_{2j_k} =1$, $k = 1 \ldots d$, $\supp
x_k\cap E_k=\emptyset$ for all $k=1,...,d$ and

\begin{equation}
\| \sum_{k=1}^d b_k x_k \| \leq \frac{C}{m^2_{2j+1}} \label{star}
\end{equation}

\n for some universal constant $C \geq 1$. On the other hand, we
may also find $(y_k)_{k=1}^d$ such that $\supp y_k \subset E_k$
and

\begin{equation}
\| \sum_{k=1}^d b_k y_k \| \geq \frac{\theta}{m_{2j+1}},
\label{2star}
\end{equation}

\n for some predetermined $\theta > 0$. The fact that these
estimates differ by a factor of $m_{2j+1}$ is the critical point
and their existence in every block subspace yields the fundamental
properties for the operators.

This work includes the crystallization of the methods for
evaluating norms in the fully modified mixed Tsirelson setting. It
is important to note that in this case the {\em basic inequality},
an important ingredient in previous constructions which reduces
the complexity, is not, and perhaps cannot be used. Some of our
techniques can be traced to earlier papers
(\cite{AB1,ADKM,ADM,DM,Ga}). In addition to the complexity inherit
in dealing with the modified structure, complications arise
related to the asymptotic $\ell_2$ structure. One should take
note that our lemmas consider more global estimates on the {\em
special convex combinations} as opposed to pointwise estimates
found in the asymptotic $\ell_1$ cases.

To conclude we state two important open problems in this area. A
remarkable result of N. Tomczak-Jaegermann and R. A. Komorowski
\cite{KT} implies that our space $\xfr_{wh}$, as well as the
$2$-convexification of modified Tsirelson space contains a
(necessary weak Hilbert) subspace without an unconditional basis.
The next question was posed by P. Casazza and can be found in
\cite{ACK}.

\begin{prob}
Does there exist a weak Hilbert space which does not embed into a
space with an unconditional basis?
\end{prob}

Finally, the ultimate problem concerning the existence of singular
weak Hilbert spaces was also posed by P. Casazza.

\begin{prob}
Does there exist a hereditarily indecomposable weak Hilbert space?
\end{prob}

\section{Schreier Families and Repeated averages}

In this section we recall the definition of the Schreier families
and their modified versions. We also give the definition of $p$-special convex
combinations and use the Repeated Averages Hierarchy,
introduced in \cite{AMT}, to prove their existence.

The recursive definition of the generalized Schreier
hierarchy $(\ssh_n)_{n<\omega}$ is as follows.

\begin{defn}
\label{sch1} Set $\ssh_1=\{F\subset\N: |F|\leq \min F\}$. Let
$n\in\N$ and suppose that $\ssh_n$ has been defined. We set
$\ssh_{n+1}=\{\cup_{i=1}^d F_i:F_1<...<F_d, F_i\in\ssh_n,\text{
for } 1\leq i\leq d, \{\min F_i\}_{i=1}^d\in \ssh_1\}$.
\end{defn}

Note that for all $n\in \N$ the family $\ssh_n$ is compact,
hereditary and spreading. We also use the following
notation for the {\em convolution} of two compact, hereditary and
spreading families of finite subsets of $\N$,

\begin{notation}
\label{schnot1} Let $Q,P$ be two families of finite subsets of
$\N$. We denote by $P[Q]$ the following family:

\[P[Q]=\{\cup_{i=1}^d F_i:F_1<...<F_d, F_i\in Q \text{ for }1\leq
i\leq d,\{\min F_i\}_{i=1}^d\in P\}.\]
\end{notation}

Observe that,
$\ssh_{n+1}=\ssh_1[\ssh_n]$ and that more generally,
$\ssh_{l+k}=\ssh_l[\ssh_k]=\ssh_k[\ssh_l]$, for all $k,l\in\N$.  Throughout this article we will be using the modified versions of the
generalized Schreier families.  These are defined as follows:

\begin{defn}
\label{sch2} Set $\ssm_1=\ssh_1$. Let
$n\in\N$ and suppose that $\ssm_n$ has been defined. We set
$\ssm_{n+1}=\{\cup_{i=1}^d F_i:(F_i)_{i=1}^d\text{ are pairwise
disjoint}, F_i\in\ssm_n,\text{ for } 1\leq i\leq d, \{\min
F_i\}_{i=1}^d\in \ssh_1\}$.
\end{defn}

The analogous notation for the modified convolution,
is the following.
\begin{notation}
\label{schnot2} Let $Q,P$ be two families of finite subsets of
$\N$. We denote by $P[Q]_M$ the following family:

\[P[Q]_M=\{\cup_{i=1}^d F_i:(F_i)_{i=1}^d\text{ are pairwise
disjoint}, F_i\in Q \text{ for }1\leq i\leq d,\{\min
F_i\}_{i=1}^d\in P\}.\]
\end{notation}

Again observe that the definition of the
$(\ssm_n)_{n<\omega}$ is equivalent to setting
$\ssm_{n+1}=\ssm_1[\ssm_n]_M$ and that more generally $\ssm_{l+k}=\ssm_l[\ssm_k]_M=\ssm_k[\ssm_l]_M$, for all
$k,l\in\N$. It is proved in \cite{ADKM} that for all $n\in\N$ we have $\ssh_n=\ssm_n$.

For $E_i\subset \N$ we say $(E_i)_{i=1}^d$ is $S_{n}$-admissible
if $E_1 < E_2 < \cdots <E_d$ and $(\min E_i)_{i=1}^d \in S_{n}$. A
sequence of finite sets of $\N$, $(E_i)_{i=1}^d$, is
$S_n$-allowable if $(E_i)_{i=1}^d$ are pairwise disjoint and
$(\min E_i)_{i=1}^d \in S_n$. A sequence of vectors
$(x_i)_{i=1}^d$ is $S_n$-allowable (resp. admissible) if $(\supp
x_i)_{i=1}^d$ is $S_n$-allowable (resp. admissible).

We are ready to define the $p$-special convex combinations.

\begin{defn}
\label{bscc} Let $n\in\N$, $F\in\ssh_n$, $p\geq 1$, $\e>0$
and $x=\sum_{k\in F}b_ke_k\in c_{00}$. The vector $x$ will be
called a $(p,\e,n)$-basic special convex combination (bscc)
if the following hold:

\begin{enumerate}
\item $(\sum_{k\in F} b_k^p)^{1/p}=1$

\item For all $k<n$ and
$G\in \ssh_k$, $(\sum_{k\in G}b^p_k)^{1/p}<\e$.
\end{enumerate}
\end{defn}

\begin{defn}
\label{scc} Let $n\in\N$, $F\in\ssh_n$, $p\geq 1$, $\e>0$
and $x=\sum_{k=1}^d b_k y_k$, such that $(y_k)_{k=1}^d$ a
block sequence in $c_{00}$. The vector $x$ will be called a
$(p,\e,n)$-special convex combination (scc) if $\sum_{k=1}^d b_k e_{t_k}$ is a $(p,\e,n)$ bscc, where $t_k=\maxsupp
y_k$, for all $k \in \{ 1, \ldots, d\}$.
\end{defn}

\begin{rem} If $\sum_{k=1}^d b_k y_k$ is a $(p,\e,n)$-scc then the following hold:
\begin{itemize}
\item[(i)] $\{ \maxsupp y_k\}_{k=1}^d \in \ssh_n$
\item[(ii)] For each $l <n$ and $G \subset \{1, \ldots, d\}$ such that $\{ \maxsupp y_k : k \in G\} \in \ssh_l$,

\n $(\sum_{k \in G} b_k^p )^{1/p} < \e$.
\end{itemize}
\end{rem}

For the sake of completeness we prove the existence of the averages
defined above.  Similar averages have been used in \cite{AB1} and \cite{DM}.  We follow
the notation of \cite{AT}.

\begin{defn}
\label{RAA} Let $L\in [\N]$. The $n$-averages, $(a_n^L)_n \in c_{00}$ are
defined recursively as follows. For $n=0$ we set $a_0^L=e_{\min{L}}$.
Suppose we have defined $a_n^M$ for some $n\in \N$ and all
$M\in[\N]$. Then set $l_1=\min L$, $L_1=L$ and
$L_k=L_{k-1}\setminus \supp a_n^{L_{k-1}}$, for $k=2,...,l_1$.
Finally,
$$a_{n+1}^L=\frac{1}{l_1}(a_n^{L_1}+...+a_n^{L_{l_1}}).$$
\end{defn}

The following properties of the RAA can be easily verified using
induction:

\begin{enumerate}
\item $\|a_n^L\|_1=1$, for all $n\in\N$ and $L\in [\N]$ and
$a_n^L(k)\geq 0$ for all $k\in\N$. \item $\supp a_n^L$ is the
(unique) maximal initial segment of $L$ that belongs to $\ssh_n$.
\end{enumerate}

The following proposition establishes the presence of
$(1,\e, n)$-bscc in $c_{00}(\N)$, for all $n\in\N$ and
$\epsilon>0$.

\begin{prop}
\label{exbscc} Let $n\in \N$, $L\in [\N]$. For all $m<n$ and $G\in
\ssh_m$ we have that
\[\sum_{k\in G}a_n^L(k)<\frac{3}{\min L}\]
\end{prop}

We refer the interested reader to \cite{AT} for a detailed proof of the above.
The existence of $p$-bscc, for $p>1$, is an immediate consequence of the following.

\begin{rem}
\label{expbscc} Let $\e>0$, $n\in\N$, $F\in\ssh_n$ and
$p>1$. If $x=\sum_{k\in F} b_k e_k$ is a
$(1,\e^p, n)$-bscc then $y=\sum_{k\in F} b_k^{1/p}e_k$ is a
$(p,\e, n)$-bscc.

\end{rem}
Using Proposition \ref{exbscc} and Remark \ref{expbscc} we can
readily establish the following.

\begin{rem}
\label{exscc}
Let $(y_k)_{k=1}^\infty$ be a block sequence in $c_{00}$, $\e > 0$ and $n \in \N$.  There is an interval $E \subset \N$ and $(b_k)_{k \in E}$ such
that  $\sum_{k \in E} b_k y_k$ is a $(\e, n)$-scc. Moreover, if $3/\maxsupp y_1< \e$ and $d \in \N$ is such that $\{\maxsupp y_k\}_{k=1}^d$ is a maximal element of $\ssh_n$, then there exists
$(b_k)_{k=1}^d \in c_{00}$ such that $\sum_{k =1}^d b_k y_k$ is a $(\e, n)$-scc.  \label{coeff}

\end{rem}

In the sequel we work only with in the case
$p=2$, so whenever we consider a $(2,\epsilon,n)$-scc, for some
$\epsilon>0$ and $n\in\N$, we shall refer to it as an
$(\epsilon,n)$-scc. \\

\section{The definition of $\xfr_{wh}$ and its basic properties}

\begin{defn}
Let $D \subset c_{00}$, $m >1$ and $n \in \N$.  We say that $D$ is closed in the modified $\ell_2-(1/m,S_n)$ operation if for every $(f_i)_{i=1}^d \subset D$ such that $(f_i )_{i=1}^d$ is $S_n$-allowable and $(\lambda_i)_{i=1}^d \in Ba(\ell_2)$ the vector

$$\frac{1}{m} \sum_{i=1}^d \lambda_i f_i \in D.$$.

\n Let $\w(f) = m$ (weight of $f$) when $f$ is the result of the above operation.
\end{defn}

The definition of the space $\xfr_{wh}$ requires that we fix two
increasing sequences of positive integers $(n_i)_{i=0}^\infty$ and
$(m_i)_{i=0}^\infty$ satisfying certain growth conditions. Let
$m_0=m_1=2$, $n_0=1$ and for $j \geq 2$ let:
\begin{enumerate}
\item $m_{j+1} \geq m_j^3$, $\ell_j = 3\log_2(m_j)+1$.
\item $n_j$ is chosen such that $\ell_j(n_{j-1} +1)<n_j$.
\end{enumerate}

\n As in the section 2, let $N_1$ and $N_2$ be infinite subsets of $\N$ with $\N=N_1\cup N_2$ and let
\begin{equation*}
\begin{split}
\Sigma = \{(E_i,2j_i)_{i=1}^n: &~ E_i \cap E_j = \emptyset \mbox{ and } j_1 < j_2 < \cdots < j_n \\
& \text{ with $j_1 \in N_1$ and $j_i \in N_2$ for $i > 1$}\}.
\end{split}
\end{equation*}

\n Define an injective coding $\sigma: \Sigma \to N_2$ such that

$$m_{2\sigma((E_1,2j_1), \ldots, (E_{i+1},2j_{i+1})) } > m_{2\sigma((E_1,2j_1), \ldots, (E_{i},2j_{i})) } \cdot (\maxsupp E_i)^2.$$

\begin{defn} {\em ($\sigma$-special sequences).
\begin{enumerate}
\item A sequence $(E_i,2j_i)_{i=1}^\infty$ is {\em $\sigma$-special} if $j_1 \in N_1$ and for each $i \geq 1$,
$$\sigma((E_1,2j_1), \ldots, (E_i,2j_i))=j_{i+1}.$$
\item A $\sigma$-special sequence $(E_i,2j_i)_{i=1}^p$ is a {\em $S_{n_{2j+1}}$ $\sigma$-special sequence} if $(\min E_i)_{i=1}^p \in S_{n_{2j+1}}$ and $2j_1 > 2j+2$.
\item $(f_i)_{i=1}^p\subset c_{00}(\N)$ is a {\em $\sigma$-special sequence of functionals} ({\em $S_{n_{2j+1}}$ $\sigma$-special sequence of functionals}) if there exists a $\sigma$-special sequence $(E_i,2j_i)_{i=1}^p$ ($S_{n_{2j+1}}$ $\sigma$-special sequence resp. ) such that $\supp g_i \subset E_i$ and $\w(f_i)=m_{2j_i}$ for each $1 \leq i \leq p$.
\end{enumerate}}
\end{defn}

It follows from the definition that the $\sigma$-special sequences
satisfy the {\em extension property} and the {\em tree-like
property} from Section 2.

\begin{defn}
The norming set $D_{wh}$ is the minimal subset of $c_{00}(\N)$ such that
\begin{enumerate}
\item $D_{wh}$ contains $\{\pm e_n^* : n \in \N \}$
\item $D_{wh}$ is closed under $\ell_2-(1/m_{2j},S_{n_{2j}})$ operations.
\item $D_{wh}$  contains $f= 1/m_{2j+1} \sum_{i=1}^p \lambda_i f_i$ such that $(\lambda_i)_{i=1}^p \in Ba(\ell_2)$ and $(f_i)_{i=1}^p$ is a  $S_{n_{2j+1}}$ $\sigma$-special sequence of functionals.
\end{enumerate}
\end{defn}

  The next lemma concerns a decomposition, or tree analysis, of each $f \in D_{wh}$.  It is routine to check that every $f \in D_{wh}$ admits such an analysis.

\begin{defn} (Tree Analysis)
Let $f \in D_{wh}$.  A {\em tree analysis} of $f$ is a set $(f_\a)_{a \in \mathcal{A}}$ such $\mathcal{A}$ is finite tree with a unique root $0 \in \mathcal{A}$ satisfying the following conditions.
\begin{enumerate}
\item $f_0 = f$ and $f_\a \in D $ for all $\a \in \mathcal{A}$.
\item $\a \in \mathcal{A}$ is terminal if  and only if $f_\a \in D_0$.
\item Let $\a \in \mathcal{A}$ be non-terminal. Denote by $S_\a$, the set of all immediate successors of $\a$ in $\mathcal{A}$.  Then there exists $j \in \N$ and  $\sum_{\beta \in S_\a} \lambda_\beta^2 \leq 1$ such that,

$$f_\a = \frac{1}{m_j} \sum_{\beta \in S_\a} \lambda_\beta f_\beta. $$

\n If $j$ is even, $(f_\b)_{\b \in S_\a}$ is $S_{n_j}$-allowable and if $j$ is odd $(f_\b)_{\b \in S_\a}$ is a $S_{n_{2j+1}}$ $\sigma$-special sequence of functionals.
\end{enumerate}
\end{defn}

\begin{notation}  Let $f \in D_{wh}$ and fix a tree analysis $(f_\a)_{\a \in \mathcal{A}}$ of $f$.
\begin{enumerate}
\item Let $m(\alpha)=\Pi_{\beta \prec \alpha} \w(f_\beta)$ for $\alpha \not=0$.
\item Let $\lambda(\alpha)=\Pi_{\beta \preceq \alpha} \lambda_{\beta}$ for $\alpha \not=0$.
\end{enumerate}
\end{notation}

\begin{defn}(The tree representation).
Let $(f_{\ell})_{\ell=1}^r\subset D_{wh}$ and
$(\l_{\ell})_{\ell=1}^r\in Ba(\ell_2)$. For each $\ell \in \{1,
\ldots ,r\}$ let $(f_\a)_{\a \in \mathcal{A}_\ell}$ be the tree
analysis of $f_\ell$ and let $\ell \in \mathcal{A}_\ell$ denote
its unique root.  Set $\mathcal{A}= \cup_{\ell =1}^r
\mathcal{A}_\ell$.  The collection $(f_\a)_{\a \in \mathcal{A}}$
is a tree representation of $\sum_{\ell =1}^r \l_\ell f_\ell$.
\end{defn}

\begin{notation} For each $\a \in
\mathcal{A}$ we let, \begin{enumerate} \item$\l(\a) = \l_\ell
\prod_{\g \preceq \a} \l_\g$. \item $m(\a) = \prod_{\g \prec \a}
\w(f_\g)$.\end{enumerate} \end{notation}

In the following Remark we state some basic properties concerning
antichains of $\mathcal{A}$ where $(f_a)_{a\in\mathcal{A}}$ is a
tree analysis of a functional $f$ or a tree representation of a
functional of the form $\sum_{\ell=1}^r \l_{\ell} f_{\ell}$.

\begin{remark}\label{rem3}
Let $(f_a)_{a\in\mathcal{A}}$ be a tree analysis or a tree
representation. If $D\subset \mathcal{A}$ is an antichain, then the
following can be readily verified:
\begin{enumerate}
\item[1.] $\sum_{\g\in D} \l^2(\g)\leq 1$. \item[2.] If $D$ is
maximal then $f=\sum_{\g\in D} \frac{\l(\g)}{m(\g)}f_{\g}$.
\item[3.] For $x\in\xfr_{wh}$, if we set $x^D=x|_{\cup_{\g\in
D}\supp f_{\g}}$, then, $f(x^D)=\sum_{\g\in D} \frac{\l(\g)}{m(\g)}f_{\g}|_{\supp x^D}(x^D)$.
\end{enumerate}
\end{remark}

At this point we state the following result which concerns the
allowability of families of functionals defined through antichains
of a given tree representation.

\begin{lem}\label{lem6} Let $j\in\N$,
$(f_{\ell})_{\ell=1}^r\subset D_{wh}$ a $S_{n_{j-1}}$ allowable
family and $(f_a)_{a\in \mathcal{A}}$ be a tree representation of the family
$(f_{\ell})_{\ell=1}^r\subset D_{wh}$ and
\[F=\{a\in \mathcal{A}: \prod_{\b\prec
a}\frac{1}{w(f_{\b})}>\frac{1}{m_j^3},~w(f_{\b})\leq m_{j-1}\text{
for all }\b\prec a\}\] Then for every antichain $G\subseteq F$ the
family $\{f_a:a\in G\}$ is $S_{n_j-1}$ allowable.
\end{lem}

\begin{proof} Let $a\in F$. It is easily verified that $|\{\b\in
\mathcal{A}:\b\prec a\}|<3\log_2 (m_j)$. The convolution property of the
modified Schreier families and the fact that for all $\b\prec a$,
$w(f_{\b})\leq m_{j-1}$ yield that for a given
$\ell\in\{1,...,r\}$ the family $\{f_{a}:a\in F\cap \mathcal{A}_{\ell}\}$ is
$\ssh_{(3\log_2 (m_j)\cdot n_{j-1})}$ allowable. Since the family
$(f_{\ell})_{\ell=1}^r$ is $\ssh_{n_{j-1}}$ allowable we have the
result.
\end{proof}

For the definition of $\xfr_{wh}$ it is now routine to check that the norm satisfies the implicit formulas stated in the introduction.
We restate the norms here for reference.
 Note that $\xfr_{wh}$ is the completion of $c_{00}$ in the norm $\|x\|=\max\{\sup \{ \|x \|_j : j \in \N\cup\{0\}\},\|x\|_{\infty}\}$ where for each $j \in \N$, $\|\cdot\|_j$ satisfy the following implicit formulas:

\begin{equation}
\begin{split}
& \label{2j}\| x \|_{2j} =\sup \bigg\{ \frac{1}{m_{2j}} \bigg( \sum_{i=1}^k \|E_i x \|^2 \bigg)^\frac{1}{2} : (E_i)_{i=1}^k \mbox{ is $S_{n_{2j}}$ allowable }\bigg\}, \\
& \|x\|_{2j+1}=\sup \bigg\{ \frac{1}{m_{2j+1}}  \bigg( \sum_{i=1}^k \|E_i x \|_{2j_i}^2 \bigg)^\frac{1}{2} : (E_i, 2j_i)_{i=1}^k \mbox{ is a $S_{n_{2j+1}}$-$\sigma$ special sequence}\bigg\}.
\end{split}
\end{equation}

We now establish the following:

\begin{prop}
The basis $(e_i)_{i\in\N}$ of $\xfr_{wh}$ is asymptotic $\ell_2$ for vectors with disjoint support and therefore $\xfr_{wh}$ is a weak Hilbert space.
\label{dsupp}
\end{prop}

\n The proof of this proposition follows from Remark \ref{ll2} and Lemma \ref{ul2}.  The next remark follows from equation (\ref{2j}).

\begin{rem}\label{ll2}
For every sequence of disjointly supported vectors $(x_k)_{k=1}^d$ such that $d \leq \supp x_i$ for $i\in \{1, \ldots d\}$,

$$\frac{1}{2} \biggl(\sum_{k=1}^d \|x_k\|^2 \biggr)^\frac{1}{2} \leq \bigg\| \sum_{k=1}^d x_k \bigg\|. $$
\end{rem}

\n The following remark is a critical, however simple, observation we use in the proof of Lemma \ref{ul2}.

\begin{rem}
Let $f=1/m_j \sum_{i=1}^p a_i f_i \in D_{wh}$ and $x \in \xfr_{wh}$.  Then there is a $g \in D_{wh}$ such that $\supp g\subseteq \supp f$ and

$$g(x) = \frac{1}{m_j} \biggl( \sum_{i=1}^p | f_i(x) |^2 \biggr)^\frac{1}{2} $$
\label{dc}
\end{rem}

\begin{proof}
If $f_i(x)=0$ for all $i \in \{1, \ldots, d\}$, let $f=g$. Otherwise observe that,

$$g=\frac{1}{m_j} \sum_{i=1}^d \frac{f_i(x)}{(\sum_{j=1}^d f_j(x)^2)^\frac{1}{2} } f_i$$

\noindent does the trick.
\end{proof}

\begin{lem}
Let $f \in D_{wh}$ and $(x_k)_{k=1}^d$ be a sequence of disjointly supported vectors.  There exists a disjointly supported sequence $(g_k)_{k=1}^d \subset D_{wh}$ with $\supp g_k \subset \supp f$ for all $k\in \{1, \ldots d\}$ such that

$$f\bigg(\sum_{k=1}^d x_k \bigg) \leq \biggl( \sum_{k=1}^d |g_k (x_k)|^2 \biggr)^\frac{1}{2}. $$

\n Moreover, the standard basis of $\xfr_{wh}$ satisfies and upper $\ell_2$ estimate on vectors with disjoint support. \label{ul2}
\end{lem}

\begin{proof} The moreover statement follows immediately from the first conclusion since $|g_k(x_k)| \leq \|x_k\|$.
We proceed by induction on the height of the tree analysis of $f \in D_{wh}$.  If $f= \pm e_n^*$ the claim is obvious.  Assume the conclusion holds for all $k < n$ and let $f$ have a tree analysis of height $n$ with $f= 1/m_j \sum_{i=1}^p \l_i f_i $.  By applying the induction hypothesis for each $f_i$ and then the Cauchy-Schwartz inequality we have,

$$\frac{1}{m_j} \sum_{i=1}^p \l_i f_i \bigg( \sum_{k=1}^d x_k\bigg) \leq \sum_{i=1}^p \l_i \biggl( \sum_{k=1}^d |g_{k,i} (x_k) |^2 \biggr)^\frac{1}{2} \leq \frac{1}{m_j} \biggl( \sum_{k=1}^d \sum_{i=1}^p |g_{k,i} (x_k)|^2 \biggr)^\frac{1}{2}.$$ \\

\noindent For each $k\in \{1,\ldots d\}$ and any  $(a_i)_{i=1}^p \in Ba(\ell_2)$, $1/m_j \sum_{i=1}^d a_i g_{k,i} \in D_{wh}$.    Now apply Lemma \ref{dc} for each $k=1,\ldots d$, to find  $g_k \in D$ with $\supp g_k \subset \cup_{i=1}^d \supp g_{k,i}$ such that,

$$|g_k(x_k)| \leq \bigg( \sum_{i=1}^n |g_{k,i}(x_k)|^2 \bigg)^{1/2}.$$

\n The result follows.
\end{proof}

\section{Rapidly increasing sequences, exact vectors and 0-dependent sequences}

  In this section we state relevant definitions that lay the groundwork for Section \ref{operators} where we prove that the operators have the desired decomposition.
The following definition is a fundamental component in the
construction of all spaces with few operators.
\begin{defn}  Let $C>0$ and $(2j_k)_k$ be an increasing sequence of positive integers.  A block sequence $(x_k)_{k \in \N}$ in $\xfr_{wh}$ is a $(C,(2j_k)_k)$-rapidly increasing sequence (RIS), if the following are satisfied for each $k$:
\begin{enumerate}
\item $\|x_k \| \leq C$; \item Each $x_k$ is a
$(1/m_{2j_k}^3,n_{2j_k})$-scc; \item $(\maxsupp
x_k)^2/m_{2j_{k+1}} < 1/m_{2j_k}$.
\end{enumerate}
\end{defn}
\begin{defn}
A block sequence $(x_k)_{k \in \N}$ is a seminormalized
$(C,(2j_k)_{k=1}^\infty)$-RIS if it is a $(C,(2j_k))$-RIS and
$\|x_k\| \geq 1$ for all $k \in \N$.
\end{defn}
Finding a RIS is done inductively. The next lemma establishes the
existence of seminormalized scc's in every block subspace.

\begin{lem}
Let $(x_k)$ be a normalized block sequence in $\xfr_{wh}$, $\e >0$ and $j \geq 2$.  Then there exists a block sequence $(y_k)_k$ of $(x_k)_k$ such that $\|y_k\| \leq 1$
and a $(\e, n_{2j})$-scc $\sum_{k = 1}^da_k y_k$ such that $\|\sum_{k = 1}^d a_k y_k\| \geq 1/2$.
\label{semiscc}
\end{lem}

The proof of the preceding Lemma is identical to Lemma 4.5 in
\cite{ADM}, making the obvious modifications to accommodate the
$\ell_2$ structure. The next definition is new.

\begin{defn} (Exact Vector) Let $(x_k)_{k \in \N}$ be a seminormalized $(C,(2j_k)_{k=1}^\infty)$-RIS.  A vector
$x = m_{2j} \sum_{k \in E} b_k x_k $ is a $(C,2j)$ exact vector with respect to $(x_k)_{k \in \N}$ if $j < j_{\min E}$ and
$\sum_{k \in E} b_k x_k $ is a $(1/m^3_{2j}, n_{2j})$-scc.
\end{defn}

\begin{defn} (Exact Sequence) Let $(y_s)$ be a seminormalized $(C,(2i_s))$-RIS.  A block sequence $(x_k)$ is a $(C,(2j_k))$
exact sequence  with respect to $(y_s)$ if
\begin{enumerate}
\item Each $x_k$ is a $(C,2j_k)$ exact vector with respect to
$(y_s)$; \item $\minsupp x_k \leq \minsupp y_s$ implies $j_k <
i_s$; \item $\maxsupp y_s < \minsupp x_k$ implies $j_k > i_s$.
\end{enumerate}
A block sequence $(x_k)_{k \in \N}$ is a $(C, (2j_k))$ exact sequence if it is an exact sequence with respect to some RIS.
\end{defn}

The existence of such a sequence in every block subspace is straightforward.

Next we define the crucial notion that allows us to establish the
desired decomposition of operators on $\xfr_{wh}$.

\begin{defn}
Let $j \in \N$, $C>1$.  A block sequence $(x_k)_{k \in \N}$ is a $(0,C,2j+1)$ dependent sequence if
\begin{enumerate}
\item $(x_k)$ is a $(C,(2j_k))$ exact sequence.
\item $j+1 < j_1 \in N_1$ and there is a sequence $(E_i)_{i=1}^\infty$ such that $(E_i,2j_i)_{i =1}^\infty$ is $\sigma$-special with
$(\cup_{i \in \N} E_i) \cap (\cup_{i \in \N } \supp x_i ) = \emptyset$ and $\maxsupp x_i < \maxsupp E_i$ for all $i \in \N$.
\end{enumerate}
In this case we say that $(x_k)$ is a $(0,C,2j+1)$ dependent sequence with respect to $(E_i,2j_i)_{i=1}^\infty$.
\end{defn}

\begin{rem}
If $(x_k)_{k \in \N}$ is a $(0,C,2j+1)$ dependent sequence then $(\frac{1}{m_{2j_k}}x_k)$ is a $(C,(2j_k))$-RIS.
\end{rem}

\begin{proof}
Conditions (1) and (2) in the definition of RIS are immediate.  Since $\sigma((E_i,2j_i)_{i=1}^k)=j_{k+1}$ for all $k \in \N$ and
$$m_{2j_{k+1}}>m_{2j_{k}}(\maxsupp E_k)^2 > m_{2j_k} (\maxsupp x_k)^2$$
condition $(3)$ is verified.
\end{proof}

\section{ The Basic Evaluation }\label{MP}

The aim of this section is to show the following result:

\begin{prop}\label{7.10} Let $j \in \N$, $C>0$, and  $(x_k)_{k =1}^d$ be a $(0,C,2j+1)$ dependent sequence and let $\sum_{k =1}^d b_k x_k$ be a $(1/m^2_{2j+2}, n_{2j+1})$-scc and $f\in D_{wh}$ with $w(f)<m_{2j+1}$. Then,

$$f( \sum_{k =1}^d b_k x_k ) \leq \frac{C}{m^2_{2j+1}} .$$
\end{prop}

In order to prove the above we need some preliminary work.

We begin by providing three fundamental techniques that allow us
to derive small estimates.

\begin{lem}
\label{mfe} [Minimum Free Estimate (MFE)] Suppose
$(f_\ell)_{\ell=1}^p \subset D_{wh}$ is a sequence of disjoint
functionals and $(x_k)_{k=1}^d$ is a block sequence in $\xfr_{wh}$
with $\|x_k\| \leq C$ such that for all $\ell=1, \ldots, p$ and
$k=1,\ldots, d$, $\minsupp f_\ell \not\in \ran x_k$.  For
$(\l_\ell)_{\ell =1}^p  \in Ba(\ell_2)$ and $(b_k)_{k=1}^d \in
c_{00}$,
$$\sum_{\ell =1}^p \l_\ell f_\ell \bigg( \sum_{k=1}^d b_k x_k \bigg) \leq 4 C\bigg(\sum_{k=1}^d b_k^2  \bigg)^\frac{1}{2}.$$ \label{before}
\end{lem}
The above Lemma although easy is critical. It is similar to Lemma
4.9 in \cite{ADM} and the key point is that it distinguishes the
behavior of modified mixed Tsirelson spaces $T_M[(\theta_n,
\ssh_n)_n]$ built on the Schreier families from the ones of the
form $T_M[(\theta_n, A_n)_n]$ which use the lower complexity
families $(A_n)_n$. Indeed, as it is known, the latter class
(including among others the modified Schlumprecht space) contains
members with subspaces isomorphic to $\ell_1$ and hence not
reflexive (c.f. \cite{Ma}).

\begin{proof}
Let $I_0= \emptyset$ and for each $k\in \{1,\ldots, d\}$, define
$I_k = \{\ell : \minsupp f_\ell < \supp x_k \}$.  By definition,
$\{ f_{\ell}|_{[\minsupp x_k , \infty)} : \ell \in I_k \}$ is
$S_1$-allowable.  For each $t=1,\ldots, d$, define $\xi_t^2 =
\sum_{\ell \in I_t\setminus I_{t-1}} \l_\ell^2$, $g_t = (1/2)
\sum_{\ell \in I_t \setminus I_{t-1}}(\l_\ell/\xi_t)
f_\ell|_{[\minsupp x_k , \infty)} \in D_{wh}$ and $E_t =\supp
g_t$.  Applying the upper $\ell_2$ estimate, the Cauchy-Schwartz
inequality and the the lower $\ell_2$ estimate, it follows that,

\begin{equation*}
\begin{split}
\sum_{t=1}^d \sum_{\ell \in I_t \setminus I_{t-1}} & \g_\ell f_\ell \bigg( \sum_{k =1}^d b_k x_k \bigg) = 2 \sum_{t=1}^d \xi_t g_t \bigg( \sum_{k=t}^m b_k x_k \bigg) \leq 2 \sum_{t=1}^d \xi_t \bigg(\sum_{k=t}^d \|E_t(b_k x_k)  \|^2 \bigg)^\frac{1}{2} \\
&  \leq 2 \bigg(\sum_{k=1}^d \sum_{t=1}^k  \|E_t(b_k x_k)  \|^2
\bigg)^\frac{1}{2} \leq 4 \bigg(\sum_{k=1}^d b_k^2 C^2
\bigg)^\frac{1}{2} \leq 4C\bigg(\sum_{k=1}^d b_k^2
\bigg)^\frac{1}{2}
\end{split}
\end{equation*}
\end{proof}

\begin{lem}\label{sae}[Small Allowability Estimate (SAE)].Let $(f_{\ell})_{\ell=1}^r\subset D_{wh}$ be $S_q$ allowable, and $(x_k)_{k=1}^d$ a block sequence
in $\xfr_{wh}$ with $\|x_k\|\leq C$ for $k=1,...,d$ such that
$\sum_{k=1}^d b_k x_k$ is a $(\epsilon,n_{j_0})$ scc with
$n_{j_0}>q$ and $\epsilon\leq\frac{1}{m_{j_1}^2}$, where $j_1=j_0$
if $j_0$ is even and $j_1=j_0+1$ if $j_0$ is odd. If we set
\[\Phi=\{k\in\{1,...,d\}: \exists \ell\in\{1,...,r\} \text{ with } \minsupp f_{\ell}\in\ran x_k \}\]
then for $(\l_{\ell})_{\ell=1}^r\in Ba(\ell_2)$,
\[\sum_{\ell=1}^r \lambda_{\ell} f_{\ell}(\sum_{k=1}^d b_k x_k)\leq \frac{C}{m_{j_0}}.\]
\end{lem}

\begin{proof}
We can readily observe that $\frac{1}{m_{j_1}}\sum_{\ell=1}^r
\lambda_{\ell} f_{\ell}\in D_{wh}$ and also the family $\{\maxsupp
x_k: k\in \Phi\}$ is $S_q$ allowable. Hence, an application of the
upper $\ell_2$-estimate yields,
\begin{equation*}
\begin{split}
\sum_{\ell=1}^r \l_{\ell} f_{\ell} & = m_{j_1}(\frac{1}{m_{j_1}}\sum_{\ell=1}^r \lambda_{\ell} f_{\ell})(\sum_{k\in\Phi} b_k x_k)\\
& \leq C m_{j_1} (\sum_{k\in\Phi}b_k^2)^{1/2} \leq C
m_{j_1}\cdot\epsilon<\frac{C}{m_{j_1}}.
\end{split}
\end{equation*}

\end{proof}

Our third and final basic estimate is the one that utilizes the
RIS condition.

\begin{lem}\label{rise} [RIS estimate (RISE)]. Let $(x_k)_{k=1}^d$ be a block sequence and $(m_{2j_k})_{k=1}^{d+1}$ a sequence of
weights such that $\frac{(\minsupp x_k)^2}{m_{2j_{k+1}}}\leq
\frac{1}{m_{2j_k}^2}$. Then, for every disjoint family
$(f_{\ell})_{\ell=1}^r\subset D_{wh}$ satisfying
$\|f_{\ell}|_{\supp x_k}\|_{\infty}\leq \frac{1}{m_{2j_{k+1}}}$
for all $\ell=1,...,r$ and $k=1,...,d$ and $(c_k)_{k=1}^d\in
c_{00}$, with $0\leq c_k\leq m_{2j_k}$  we have
\[\sum_{\ell=1}^r f_{\ell}(\sum_{k=1}^d c_k x_k)\leq \frac{2}{m_{2j_1}}.\]
\end{lem}

\begin{proof}

Let $k\in\{1,...,d\}$. Observe that as the functionals
$(f_{\ell})_{\ell=1}^r$ have disjoint supports, $|\{\ell:\minsupp
f_{\ell}\leq \maxsupp x_k\}|\leq \maxsupp x_k$. Hence,

\[\sum_{\ell=1}^r f_{\ell}(x_k)\leq \sum_{\ell=1}^r \|f_{\ell}|_{\supp x_k}\|_{\infty}\cdot \maxsupp x_k\leq \frac{(\minsupp x_k)^2}{m_{2j_{k+1}}}\leq \frac{1}{m_{2j_k}^2}\]

Summing up for all $k$ yields the result.

\end{proof}
 The next lemma concerns that application of a $\ell_2$
convex combination of a sequence of functionals with `small'
allowability to a special convex combination of vectors with
comparatively large character. The term `character' refers to the
$n$ component of a $(\e,n)$-scc.

\begin{lem}\label{7.2}
Let $C \geq 1$ and  $(x_k)_{k =1}^d$ be a block sequence in
$\xfr_{wh}$ such that $\|x_k\| \leq C$.  Let $\sum_{k =1}^d b_k
x_k$ be a $(\e,n_{j})$-scc with $\e \leq 1/m^2_{j_1}$, where
$j_1=j$ if $j$ is even and $j_1=j+1$ otherwise. Let also
$(f_\ell)_{\ell =1}^r$ be $\ssh_q$-allowable for $q < n_{j}$ and
$(\l_\ell)_\ell \in Ba(\ell_2)$,.  Then,

$$\sum_{\ell =1}^r \l_\ell f_\ell \bigg(\sum_{k=1}^d b_k x_k \bigg) \leq 5C.$$
\end{lem}

\begin{proof}
Let $\Phi=\{k\in\{1,...,d\}: \exists \ell\in\{1,...,r\} \text {
with } \minsupp f_{\ell}\in\ran x_k\}$ and
$\Phi^C=\{1,...,k\}\setminus\Phi$. Then,

Applying  the SAE (Lemma \ref{sae}) we can see that
$\sum_{\ell=1}^r \l_{\ell}f_{\ell}(\sum_{k\in\Phi } b_k x_k)\leq
\frac{Cm_j}{m_{j_1}^2}$.\\
 By the MFE (Lemma \ref{mfe})
$\sum_{\ell=1}^r \l_{\ell}f_{\ell}(\sum_{k\in\Phi^C } b_k x_k)\leq
4C$.

Combining the above yields the result.
\end{proof}

The next lemma is more general than the previous in some sense,
but is not a direct generalization.  Here we assume that the
disjoint sequence of functionals is at most $\ssh_{n_q}$ allowable
for $n_q$ strictly smaller that the characters of a sequence of
scc's. In the previous lemma we assumed the allowability was of
the sequence of functionals at at most {\bf one} less that the
character of one scc.  The major difference (or restriction) of
course is that the allowability must be indexed by a member of the
sequence $(n_i)$ and this is not so in the first lemma.  The proof
is very similar.

\begin{lem}
\label{7.3} Let $C \geq 1$ and  $(x_k)_{k =1}^d$ be a block
sequence in $\xfr_{wh}$ such that $\|x_k\| \leq C$ and each $x_k$
is $(1/m^2_{2j_k},n_{2j_k})$-scc and $j_1 \leq \cdots \leq j_d$.
Let  $(f_{\ell})_{\ell=1}^p \subset D$ be $\ssh_{n_q}$-allowable
with $q< 2j_1$ and $\w(f_\ell) < m_{2j_1}$.  For
$(\l_\ell)_{\ell=1}^p\in Ba(\ell_2)$ and $(b_k)_{k=1}^d \in
c_{00}$,

$$ \sum_{\ell=1}^p \l_\ell f_\ell \bigg(\sum_{k =1}^d b_k x_k \bigg) \leq \frac{5C}{\displaystyle \min_{1\leq \ell \leq p}\w(f_\ell)} \bigg(\sum_{k =1}^d b_k^2 \bigg)^{\frac{1}{2}}.$$
\label{CL0}
\end{lem}

\begin{proof}
For each $k \in \{1, \ldots, d\}$ let $x_k= \sum_{1=1}^{p_k}
b_{k,i}x_{k,i}$.  For each $\ell\in \{1, \ldots, p\}$ denote by
$S_\ell$ the immediate successors of $f_\ell$ in its tree analysis
and let $S=\cup_{\ell =1}^p S_\ell$.  Let,
$$\Phi =\{(k,i) : \exists~ \a \in S \text{ such that, } \minsupp f_\a \in \ran x_{k,i}\}$$

\n For each $k \in \{1, \ldots, d\}$ let $\Phi_k= \{ i : (k,i) \in
\Phi\}$.  For each $\ell \in \{1 ,\ldots, p\}$ we have assumed
$(f_\a)_{\a \in S_\ell}$ is $\ssh_{n_{2j_1-1}}$-allowable.  By the
convolution property of Schreier families we may conclude that $\{
f_\a : \a \in S \}$ is $\ssh_{2n_{2j_1-1}}$-allowable.  Recall
that $2n_{2j_1-1}<n_{2j_1}$ and observe that $k \in \{1,\ldots
,d\}$ $\{\maxsupp x_{k,i} : i \in \Phi_k\} \in
\ssh_{2n_{2j_1-1}}$. Observe that,

\begin{equation*}
\sum_{\ell =1}^p \l_\ell f_\ell \bigg(\sum_{k =1}^d b_k \sum_{i
\in \Phi_k} b_{k,i} x_{k,i} \bigg) = \sum_{\ell =1}^p \l_\ell
\frac{1}{\w(f_\ell)}  \sum_{\a \in S_\ell}\l_\a f_\a \bigg(\sum_{k
=1}^d b_k \sum_{i \in \Phi_k} b_{k,i} x_{k,i} \bigg)
\end{equation*}

Applying the SAE we obtain,

\[\sum_{\ell =1}^p \l_\ell f_\ell \bigg(\sum_{k =1}^d b_k \sum_{i \in \Phi_k} b_{k,i} x_{k,i} \bigg)\leq \sum_{k=1}^d \frac{C}{m_{2j_k}}\leq \frac{2C}{m_{2j_1}}\]

For the remaining part we may apply MFE to see that,
\begin{equation*}
\begin{split}
\sum_{\ell =1}^p \l_\ell f_\ell& \bigg(\sum_{k =1}^d b_k \sum_{i \in \{1,\ldots,p_k\} \setminus \Phi_k} b_{k,i} x_{k,i} \bigg) \\
&\leq \frac{1}{\displaystyle \min_{1\leq \ell \leq p}\w(f_\ell)} \sum_{\ell =1}^p \sum_{\a \in S_\ell} \l_\ell \l_\a f_\a \bigg(\sum_{k=1}^d b_k \sum_{i \in \{1,\ldots,p_k\} \setminus \Phi_k} b_{k,i} x_{k,i} \bigg) \\
& \leq \frac{4C}{\displaystyle \min_{1\leq \ell \leq p}\w(f_\ell)}
\end{split}
\end{equation*}

Combining the above estimates yields the result.

\end{proof}

The final element we need in the proof of Proposition \ref{7.10}
is the following:

\begin{prop} \label{7.8} Let $(x_k)_{k=1}^d$ be a $(C,2j_k)_k$
exact sequence and $(f_{\ell})_{\ell=1}^r\subset D_{wh}$ be
$S_{n_q}$ allowable with $q<2j_1$. Assume further that for
$k\in\{1,...,d\}$ and $\ell\in\{1,...,r\}$ such that $\supp
f_{\ell}\cap\supp x_k\neq\emptyset$ we have $w(f_{\ell})\neq
m_{2j_k}$. Then for $(\l_{\ell})_{\ell=1}^r$, $(b_k)_{k=1}^d\in
Ba(\ell_2)$

\[\sum_{\ell=1}^r \l_{\ell}f_{\ell}(\sum_{k=1}^d b_k x_k)\leq\frac{40C}{\displaystyle\min_{1\leq\ell\leq r}w(f_{\ell})}. \]
\end{prop}

We shall present the proof of this result in the last section.
Granting this we proceed to the proof of Proposition \ref{7.10}.
Our approach in proving this result is separated into two steps.
In the first step, we use the tree analysis $(f_a)_{a\in A}$ of
$f$ to split the support of each $x_k$ into four disjoint sets and
we define the vectors $x_k^G,x_k^S,x_k^{R_1},x_k^{R_2}$. The
second step is to use the preparatory estimates presented earlier
in this section to evaluate the action of $f$ onto each
$\sum_{k=1}b_k x_k^I$ where $I$ is one of the $G,S,R_1,R_2$.

\begin{proof}(Proposition \ref{7.10}).

We start by defining the spitting of the support of each $(x_k)$,
for $k=1,...,d$.

The partitioning:

Let $k\in\{1,...,d\}$. We define the following sets.\\
$G_k =\{ t \in \A: ~t \text{ is maximal } \supp f_t \cap \supp
x_k\not=\emptyset,
 \prod_{\g \prec t}\frac{1}{w(f_\g)} \leq \frac{1}{m_{2j_{k+1}}}\}$.\\
Let $P_k$ denote the maximal $t \in \mathcal{A}$ such that $ \supp f_t \cap \supp x_k \not= \emptyset $ and $t \not\in G_k$. \\
For $t \in P_k$, let $\b_t=t$ if $w(f_{\g})<m_{2j+1}$ for all
$\g\prec t$ or $\b_t=\min\{\g\prec t: w(f_{\g})\geq m_{2j+1}\}$
otherwise.

\noindent We set $S_k=\{\b_t: t\in P_k,
\frac{1}{m(\b_t)}\leq\frac{1}{m_{2j+1}^3}\}$, $R_k=\{\b_t: t\in P_k, \frac{1}{m(\b_t)}>\frac{1}{m_{2j+1}^3}\}$.\\
We split $R_k$ into: \noindent $R_k^1=\{\b_t\in R_k:
w(f_{\b_t})=m_{2j+1}\}$, $R_k^2=\{\b_t\in R_k: w(f_{\b_t})\geq m_{2j+2}\}$.\\

Using the above sets we partition $x_k$ in the following way: Set
$x_k^G=x_k|_{\cup_{t\in G_k}\supp f_t}$,
\\$x_k^S=x_k|_{\cup_{\b_t\in S_k}\supp f_{\b_t}}$,
$x_k^{R_1}=x_k|_{\cup_{\b_t\in R^1_k}\supp f_{\b_t}}$ and
$x_k^{R_2}=x_k|_{\cup_{\b_t\in R^2_k}\supp f_{\b_t}}$.

We also set $x^I=\sum_{k=1}^d b_k x^I_k$, where $I=G,S,R_1,R_2$
and we denote by $G,S,R_1,R_2$ the union for all $k$ of the
corresponding $G_k,S_k,R_k^1,R_k^2$. In the following Lemma and
Corollary we present some easy consequences of the above
partitions.
\begin{lem}\label{rem2} The set $\{\b_t:t\in\cup_{k=1}^d P_k\}$ is an antichain of $\mathcal{A}$.
\end{lem}
\begin{proof}
Let $\b_{t_1}\neq\b_{t_2}$, with $t_1,t_2\in \cup_{i=1}^d P_k$ and
suppose that $\b_{t_1}\prec\b_{t_2}$. As $\b_{t_1}$ is not maximal
$w(f_{\b_{t_1}})\geq m_{2j+1}$. However, by definition for every
$\g\prec\b_{t_2}$ we have $w(f_{\g})<m_{2j+1}$ which is a
contradiction and the proof is complete.
\end{proof}

Combining Lemma \ref{rem2} and Lemma \ref{lem6} we have the
following:

\begin{cor} \label{cor1} For $\b_t\in S$ let
$a_t=\min\{\g\prec\b_t:\prod_{\b\prec\g}\frac{1}{w(f_{\b})}\leq\frac{1}{m_{2j+1}^2}\}$.
Then the families $\{f_{a_t}:\b_t\in S\}$, $\{f_{\b_t}:\b_t\in
R\}$ are $S_{n_{2j+1}-1}$allowable.
\end{cor}

Let us also observe that,

\[f(\sum_{k=1}^d b_k x_k)=f(\sum_{k=1}^d b_k x^G_k)+f(\sum_{k=1}^d b_k
x^S_k)+f(\sum_{k=1}^d b_k x^{R_1}_k)+f(\sum_{k=1}^d b_k
x^{R_2}_k)\]

 We shall
consider the cases given by the partitions separately. We start
with $f(\sum_{k=1}^d b_k x^G_k)$. A straightforward application of
the RISE (Lemma \ref{rise}) yields,
\[f(\sum_{k=1}^d b_k x^G_k)\leq
\frac{2}{m_{2j_1}}\leq\frac{1}{m^2_{2j+1}}.\]

We pass now to $f(\sum_{k=1}^d b_k x^S_k)$. For $\b_t\in S$ let
$a_t=\min\{\g\prec\b_t:\frac{1}{m(\g)}\leq\frac{1}{m_{2j+1}^2}\}$.
By Corollary \ref{cor1} the set $\{f_{a_t}:\b_t\in S\}$ is
$S_{n_{2j+1}-1}$ allowable and by the definition of $\b_t$ we also
have that $w(f_{a_t})<m_{2j+1}$. Additionally we can readily see
that,
\[f(\sum_{k=1}^d b_k x^S_k)=\sum_{\{a_t:\b_t\in
S\}}\frac{\l(a_t)}{m(a_t)}f_{a_t}(\sum_{k=1}^d b_k x^S_k).\]

Applying Lemma \ref{7.2} on the sum,
\[\sum_{\{a_t:\b_t\in
S\}}\l(a_t)f_{a_t}(\sum_{k=1}^d b_k x^S_k).\]

We may conclude that,

\[f(\sum_{k=1}^d b_k x^S_k)\leq \frac{1}{m_{2j+1}^2}\sum_{\{a_t:\b_t\in
S\}}\l(a_t)f_{a_t}(\sum_{k=1}^d b_k x^S_k)\leq
\frac{5C}{m^2_{2j+1}}.\]

We now consider  $f(\sum_{k=1}^d b_k x^{R_2}_k)$.\\
Let $E=\{k\leq d: \exists \b_t\in R: \supp f_{\b_t}\cap\supp
x^{R_2}_k\neq\emptyset,~ w(f_{\b_t})=m_{2j_k}\}$ and
$E^c=\{1,...,d\}\setminus E$. Observe that by Corollary \ref{cor1}
the family $\{\maxsupp x_k: k\in E\}\in S_{n_{2j+1}-1}$. A direct
application of the upper $\ell _2$ estimate yields,
\[f(\sum_{k\in E} b_k x^{R_2}_k)\leq C (\sum_{k\in E}
b_k^2)^{1/2}\leq C\frac{1}{m_{2j+1}^2}.\]

For $k\in E^c$ we can see that,

\[f(\sum_{k\in E^c}b_k x^{R_2}_k)=\sum_{\{\b_t\in
R^2\}}\frac{\l(\b_t)}{m(\b_t)}f_{\b_t}(\sum_{k\in E^c}b_k
x^{R_2}_k).\]

By the definition of the set $E^c$ for $\b_t\in R^2$ and $k\in
E^c$ if $\supp f_{\b_t}\cap \supp x_k^{R_2}\neq\emptyset$ we have
that $w(f_{\b_t})\neq m_{2j_k}$. Hence, apply Proposition
\ref{7.8} to obtain,
\[f(\sum_{k\in E^c}b_k x^{R_2}_k)\leq \frac{40C}{\min_{\b_t\in
R^2} w(f_{\b_t})}\leq \frac{C}{m_{2j+1}^2}.\]

Finally we pass to estimate $f(\sum_{k=1}^d b_k x_k^{R_1})$. For
each $\b_t\in R^1$ we denote by $S_{\b_t}$ its immediate
successors in $\mathcal{A}$. By the tree like property of special
sequences and the fact that $(x_k, \phi_k)_{k=1}^d$ is a
0-dependent sequence we can see that there exists at most one
$\g_{\b}\in S_{\b}$ and at most one $k_0\in\{1,..,d\}$ such that
$\supp f_{\g_{\b}}\cap \supp x_{k_0}\neq\emptyset$ with
$w(f_{\g_{\b}})=m_{2j_{k_0}}$. This observation yields that that
the family $\{f_{\g_{\b}}:\b\in
R^1\}$ is $S_{n_{2j+1}-1}$ allowable and that the same holds for\\
$\{\maxsupp x_k^{R_1}:\exists \b\in R^1, \g\in S_{\b} \text{ with
} w(f_{\g})=m_{2j_k}\}$.
Setting $E_1=\{k\leq d:\exists \b\in R^1,
\g\in S_{\b} \text{ with } w(f_{\g})=m_{2j_k}\}$ a direct
application of the upper $\ell_2$ estimate yields,
\[f(\sum_{k\in E_1} b_k x_k^{R_1})\leq C(\sum_{k\in E_1}
b_k^2)^{1/2}\leq \frac{C}{m_{2j+1}^2}.\]

For $k\notin E_1$

\[f(\sum_{k\notin E_1} b_k x_k^{R_1})=\sum_{\b\in
R^1}\frac{1}{m_{2j+1}}\sum_{\g\in
S_{\b}}\frac{\l(\g)}{m(\g)}f_{\g}(\sum_{k\notin E_1} b_k
x_k^{R_1}).\]

We may observe now that by the definition of the special sequences
for each $\g\in\cup_{\b\in R^1} S_{\b}$ if there exists $k\notin
E_1$ such that $\supp f_{\g}\cap \supp x_k^{R_1}\neq\emptyset$
then $w(f_{\g})>m_{2j+1}$ and $w(f_{\g})\neq m_{2j_k}$. Thus,
applying Proposition \ref{7.8} we obtain,

\[f(\sum_{k\notin E_1} b_k x_k^{R_1})\leq\frac{40C}{\displaystyle
\min_{\g\in \cup_{\b}S_b}w(f_{\g})}\leq \frac{C}{m_{2j+1}^2}.\]

\end{proof}

\begin{prop} \label{norm}Let $j \in \N$, $C>0$, and  $(x_k)_{k =1}^d$ be a $(0,C,2j+1)$
dependent sequence and let $\sum_{k =1}^d b_k x_k$ be a
$(1/m^2_{2j+2}, n_{2j+1})$-scc. Then,

$$\| \sum_{k =1}^d b_k x_k \| \leq \frac{C}{m^2_{2j+1}} .$$
\end{prop}

\begin{proof} Fix a $\ssh_{n_{2j+1}}$
$\sigma$-special sequence $(\phi_k)_{k=1}^d$ such that
$m_{2j_k}=w(\phi_k)$ for $k=1,...,d$ and let $f\in D_{wh}$. Let
$(E_k,2_{j_{k}})_{k=1}^d$ be the $\ssh_{n_{2j+1}}$
$\sigma$-special sequence corresponding to $(\phi_{k})_{k=1}^d$.
If $w(f)<m_{2j+1}$, Proposition \ref{7.10} yields the result.
Assuming that $w(f)\geq m_{2j+1}$ we distinguish the following
cases:
\begin{enumerate}
\item[Case 1.] $w(f)\geq m_{2j+2}$. Assume that there exists a
$k_0\in \{1,...,d\}$ such that $w(f)=m_{2j_{k_0}}$. Then,

\[f(\sum_{k=1}^d b_k x_k)=f(b_{k_0}x_{k_0})+f(\sum_{k\neq k_0} b_k
x_k).\]

By Proposition \ref{7.8} $f(\sum_{k\neq k_0} b_k x_k)\leq
\frac{40C}{w(f)}$. Additionally, $f(b_{k_0} x_{k_0})\leq
\|x_{k_0}\||b_{k_0}|\leq\frac{C}{m^2_{2j+2}}$. Hence,
\[f(\sum_{k=1}^d b_k x_k)\leq \frac{C}{m^2_{2j+1}}.\]

\item[Case 2.] $w(f)=m_{2j+1}$. Then
$f=\frac{1}{m_{2j+1}}\sum_{\ell=1}^r \l_{\ell}f_{\ell}$ with
$(f_{\ell})_{\ell=1}^r$ being a $S_{n_{2j+1}}$ $\sigma$-dependent
sequence. Let $(F_{\ell},2_{j_{\ell}})_{\ell=1}^r$ be the
$\ssh_{n_{2j+1}}$ $\sigma$-special sequence corresponding to
$(f_{\ell})_{\ell=1}^r$. If there exists $k\in\{1,...,d\}$ such
that $\supp f\cap\supp x_k \neq\emptyset$ then by the tree like
property of special sequences there exists $\ell_0\in\{1,...,r\}$
such that $F_{\ell}=E_{\ell}$ for all $\ell<\ell_0$,
$w(f_{\ell_0})=w(\phi_{\ell_0})=m_{2j_{\ell_0}}$ and
$w(f_{\ell})\neq m_{2j_k}$ for all $\ell>\ell_0$ and $k>\ell_0$.
Since $(x_k)_{k=1}^d$ is 0-dependent we have
$(\cup_{\ell=1}^{\ell_0-1}\supp f_{\ell})\cap (\cup_{k=1}^d\supp
x_k)=\emptyset$. In this context we write,
\begin{equation*}
\begin{split}
f(\sum_{k=1}^d b_k x_k) & = \frac{1}{m_{2j+1}}\l_{\ell_0}f_{\ell_0}(\sum_{k=1}^d b_k x_k)+\frac{1}{m_{2j+1}}\sum_{\ell>\ell_0}\l_{\ell}f_{\ell}(\sum_{k=1}^d b_k x_k)=\\
& =
\frac{1}{m_{2j+1}}\l_{\ell_0}f_{\ell_0}(b_{\ell_0}x_{\ell_0})+\frac{1}{m_{2j+1}}\l_{\ell_0}f_{\ell_0}
(\sum_{k\neq\ell_0} b_k
x_k)+\frac{1}{m_{2j+1}}\sum_{\ell>\ell_0}\l_{\ell}f_{\ell}(\sum_{k=1}^d
b_k x_k).
\end{split}
\end{equation*}
For the first term, since $\sum_{k=1}^d b_k x_k$ is a
$(\frac{1}{m_{2j+2}^2},n_{2j+1})$ scc we have,
$|b_{\ell_0}|<\frac{1}{m_{2j+2}^2}$. Therefore,
$\frac{1}{m_{2j+1}}\l_{\ell_0}f_{\ell_0}(b_{\ell_0}x_{\ell_0})\leq
\frac{C}{m^2_{2j+2}}$. For the last two terms we apply Proposition
\ref{7.8} to see that
\[\frac{1}{m_{2j+1}}\l_{\ell_0}f_{\ell_0} (\sum_{k\neq\ell_0} b_k
x_k)\leq \frac{40C}{w(f_{\ell_0})}\leq \frac{C}{m^2_{2j+1}},\]

\n and
\[\frac{1}{m_{2j+1}}\sum_{\ell>\ell_0}\l_{\ell}f_{\ell}(\sum_{k=1}^d b_k
x_k)\leq \frac{40C}{\displaystyle
\min_{\ell>\ell_0}w(f_{\ell})}\leq \frac{C}{m_{2j+1}^2}.\]

\end{enumerate}
\end{proof}


\section{The space of bounded linear operators}\label{operators}

In this section we investigate the behavior of the operators in
$\mathcal{L}(Y)$ where $Y$ is block subspace of $\xfr_{wh}$. In
particular we show that every $T\in\mathcal{L}(Y)$ takes the form
$T=D+S$ where $D$ is diagonal and $S$ strictly singular. For that
purpose we start by fixing $Y$ to be a subspace of $\xfr_{wh}$
generated by a normalized  block sequence $(y_n)_n$. We start with
the following easy Remark.

\begin{rem}
\label{op0}
Let $(x_n)_n$ be a $(C,2j_k)_k$-RIS and $(B_n)_n$ be finite subsets of
$\N$ such that $B_n\subset\supp x_n$ for all $n\in\N$. Then, the sequence $B_nx_n$ is a $(C,2j_k)_k$-RIS.
\end{rem}

\begin{lem} \label{op1}
Let $(x_n)_n$ be RIS in $Y$ and $T:Y\to Y$ a bounded linear
operator. Then for every $n$ and for every partition of $\supp
x_n$ into sets $C_n,B_n$, $\lim_n C_nTB_n(x_n)=0$.

\end{lem}
\begin{proof}

Suppose, towards a contradiction, that the conclusion fails. Then,
by passing to a subsequence if necessary, we assume that there
exists $\epsilon>0$ such that $\|C_nTB_n\|>\epsilon$ for every
$n\in\N$. By Remark \ref{op0} $(B_nx_n)$ is a $(C,(2i_n))$-RIS.
For each $n \in \N$ let $f_n \in D_{wh}$ such that
$f_n(C_nTB_nx_n) >\e$ and $\supp f_n \subset C_n $.  Choose $j \in
\N$ such that $\frac{1}{m_{2j+1}}< \frac{\e}{\|T\|C}$. Our goal is
to construct sequences $(z_k)_{k=1}^\infty$, $(g_k)_{k=1}^\infty$
and $(E_i,2j_i)_{i=1}^\infty$ such that,
\begin{enumerate}
\item $(z_k)_{k=1}^\infty$ is a $(0,C,2j+1)$ dependent sequence with respect to $(E_i,2j_i)_{i=1}^\infty$;
\item $3/\maxsupp z_1 < 1/m_{2j+2}^2$;
\item $(g_k)_{k=1}^\infty \subset D_{wh}$ is $\sigma$-special with respect to $(E_i,2j_i)_{i=1}^\infty$;
\item $g_k(Tz_k) \geq \e$ for all $k \in \N$.
\end{enumerate}
Assuming we can construct these sequences we arrive at a contraction in the following way: Find $d \in \N$ such
that $(\min E_i)_{i=1}^d$ is a maximal element of  $\ssh_{n_{2j+1}}$. Using Remark \ref{exscc}
there is a sequence $(b_k)_{k=1}^d$ such that $\sum_{k=1}^d b_k z_k$ is a $(1/m^2_{2j+2},n_{2j+1})$-scc.  Indeed this is possible since
$\minsupp E_i \leq \maxsupp z_i$ for all $i \in \N$. Using the conditions on the sequences and Propostion \ref{norm} the contradiction to our choice of $j$ is as follows,
$$\frac{\e}{m_{2j+1}} < \frac{1}{m_{2j+1}} \sum_{k =1}^d b_k g_k \bigg(\sum_{k=1}^d b_k T z_k \bigg) \leq \bigg\|T \bigg(\sum_{k=1}^d b_k z_k \bigg)\bigg\| \leq \frac{C\|T\|}{m^2_{2j+1}}.$$

Let us now construct the desired sequences.  Let $j+1 < j_1 \in N_1$.  Using Remark \ref{coeff} find $F_1 \subset \N$ and $(a_{1,n})_{n \in F_1}$
such that $\sum_{n \in F_1} a_{1,n} B_nx_n$ is a $(1/m_{2j_1}^3, n_{2j_1})$-scc and $3/\maxsupp x_{\max F_1} < 1/m^2_{2j+2}$.  Set,
$$z_1 = m_{2j_1} \sum_{n \in F_1} a_{1,n} B_nx_n ~\text{ and } ~ g_1= \frac{1}{m_{2j_1}}\sum_{n \in F_1} a_{1,n} f_n.$$
Notice that $\supp g_1 \cap \supp z_1 =\emptyset$ and $g_1(T z_1) > \e$.  Let $E_1=\supp g_1 \cup \{\maxsupp z_1+1\}$ and
$j_2 =\sigma(E_1,2j_1)$.  Find $F_2> \max F_1+1$ and $(a_{2,n})_{n \in F_2}$ such that $\sum_{n \in F_2} a_{2,n} B_n x_n$ is a
$(1/m_{2j_2}^3, n_{2j_2})$-scc. Set,
$$z_2 = m_{2j_2} \sum_{n \in F_2} a_{2,n} B_nx_n ~\text{ and } ~ g_2= \frac{1}{m_{2j_2}}\sum_{n \in F_2} a_{2,n} f_n.$$
Notice that $(\supp g_1 \cup \supp g_2 ) \cap (\supp z_1 \cup \supp z_2) =\emptyset$. Let $E_2=\supp g_2 \cup \{\maxsupp z_2+1\}$.
By continuing in this manner we construct the desired sequences.  Notice that since $\maxsupp z_k < \maxsupp E_k$ and
$(\cup_k E_k) \cap (\cup_k \supp x_k )=\emptyset$, we have that $(z_k)$ is a $(0,C, 2j+1)$ dependent sequence with respect
to $(E_i, 2j_i)_{i=1}^\infty$.
\end{proof}
The following can be readily verified.

\begin{rem}\label{supremum} Let Let $(x_n)_n$ be RIS in $Y$ and $T:Y\to Y$ a bounded linear
operator. Set $s_n=\sup\|C_nTB_n x_n\|$, for $n\in \N$, where the supremum is taken over all partitions
$(B_n,C_n)$ of $\supp x_n$. Then $\lim_n s_n=0$.
\end{rem}

The next step is to show that every diagonal free operator $T:Y\to
Y$ has the property that $Tx_n\to 0$ for every RIS sequence $x_n$
in $Y$. By diagonal free we mean that $y^*_n(Ty_n)=0$ for every
$n\in \N$.

To prove this result we will need a preparatory lemma that uses a
simple counting argument and is due to W.T. Gowers and B. Maurey
\cite{GM2}. Its present form is taken from Proposition 9.3 in
\cite{ALT}. Before we state the lemma let us fix some notation.

Let $T: Y\to Y$ be a bounded linear operator and suppose that
$(x_n)_n$ is a block sequence. For each $n\in N$ we define the
following:

\begin{enumerate}
\item $A_n=\supp x_n$.

\item $P_n=\{(B,C): B\cup C=A_n, B\cap C=\emptyset,
\#B=\frac{\#A_n}{2}\}$, if $\#A_n$ is even \item $P_n=\{(B,C):
B\cup C=A_n, B\cap C=\emptyset, |\#B-\#C|=1\}$, if $\#A_n$ is odd
\item We set $L_n$ to be the entire part of $\frac{\#A_n}{2}$.
\end{enumerate}

\begin{lem} \label{counting} Let $T\in \L (Y)$ be diagonal free and $(x_n)_n$ be a block sequence in $Y$.
Then we have the following:
\begin{enumerate}
\item [i.] $A_nT
x_n=\frac{2L_n(2L_n-1)}{L_n^2}\frac{1}{\#P_n}\sum_{(B,C)\in
P_n}BTCx_n$, if $\#A_n$ is even. \item [ii.]  $A_nT
x_n=\frac{2L_n(2L_n+1)((L_n+1)^2+1)}{(L_n^2+1)(L_n+1)^2}\frac{1}{\#P_n}\sum_{(B,C)\in
P_n}BTCx_n$, if $\#A_n$ is odd.
\end{enumerate}
\end{lem}
\begin{proof}
We fix $n\in \N$. We give the proof only in the case where $\#A_n$
is even, as the other case is similar. We can write $x_n$ as
$x_n=\sum_{k=1}^d a_k y_k$. Then $A_n Tx_n=\sum_{i\in
A_n}(\sum_{k=1}^d a_k y^*_i(Ty_k))y_i$. As the operator is
diagonal free we can rewrite the above sum as:
\[A_n Tx_n=\sum_{i\in
A_n}\sum_{k\neq i} a_k y^*_i(Ty_k)y_i.\]

We fix $i\in A_n$ and pass to show that

\[(\sum_{k\neq i} a_k y^*_i(Ty_k))=\frac{2L_n(2L_n-1)}{L_n^2}\frac{1}{\#P_n}\sum_{(B,C)\in
P_n}y^*_i(BTCx_n).\]

For a fixed pair $(B,C)\in P_n$ we have that
$y^*_i(BTCx_n)=\sum_{k\in C}a_ky^*_iT(y_k)$ which is non zero only
if $i\notin C$, as $B,C$ form a partition of $A_n$ and $T$ is
diagonal free. This indicates that for each $k\neq i$ the term
$a_k y^*_i(Ty_k)$ appears in the sum $\sum_{(B,C)\in
P_n}y^*_i(BTCx_n)$ as many times as is the cardinality of the set
$R_n=\{C\subset A_n: i\notin C, k\in C \text{ and } \#C=\frac{\#
A_n}{2}=L_n\}$. We can easily see that
\[\#R_n=\frac{(2L_n-2)!}{((L_n-2)!)^2}=\frac{2L_n(2L_n-1)}{L_n^2}\cdot\frac{1}{\#P_n}.\]

This completes the proof.

\end{proof}

We can now show the following.

\begin{prop}\label{ristozero}

Let $T:Y\to Y$ be a diagonal free bounded linear operator. Then
for every $(C,2j_k)_k$-RIS $(x_n)_n$ in $Y$, $T(x_n)\to 0$. In
particular, $T$ is strictly singular.
\end{prop}

\begin{proof}

Assume that the conclusion fails and by passing to a subsequence
if necessary we suppose that $\|Tx_n\|>\epsilon$ for all $n\in
\N$.

First we observe that $\lim_n A_n T(x_n)=0$. Indeed, by Lemma \ref{counting} for each $n\in \N$ we can write $A_nT
x_n=\lambda_n\frac{1}{\#P_n}\sum_{(B,C)\in P_n}BTCx_n$, where
$1\leq\lambda_n\leq 4$. Therefore using Remark \ref{supremum} we obtain
$\lim_n A_n T(x_n)=0$. Granting this and using a sliding hump
argument we may assume the following two properties concerning the
sequences $(x_n)_n$ and $(Tx_n)$:

\begin{enumerate}
\item The sequence $(Tx_n)_n$ is a block sequence.

\item $\supp x_n\cap \supp Tx_n=\emptyset$, for all $n\in\N$.
\end{enumerate}

We choose a $j\in \N$ such that
$\frac{1}{m_{2j+1}}<\frac{\epsilon}{C\|T\|}$.  Following the same lines as in the proof of Lemma \ref{op1} we inductively
construct sequences $(z_k)_{k=1}^\infty$,
$(g_k)_{k=1}^\infty$ and $(E_i,2j_i)_{i=1}^\infty$. such that,
\begin{enumerate}
\item $(z_k)_{k=1}^\infty$ is a $(0,C,2j+1)$ dependent sequence with respect to $(E_i,2j_i)_{i=1}^\infty$;
\item $3/\maxsupp z_1 < 1/m_{2j+2}^2$;
\item $(g_k)_{k=1}^\infty \subset D_{wh}$ is $\sigma$-special with respect to $(E_i,2j_i)_{i=1}^\infty$;
\item $g_k(Tz_k) \geq \e$ for all $k \in \N$.
\end{enumerate}
Granting this find $d \in \N$ such
that $(\min E_i)_{i=1}^d$ is a maximal element of  $\ssh_{n_{2j+1}}$. Using Remark \ref{exscc}
there is a sequence $(b_k)_{k=1}^d$ such that $\sum_{k=1}^d b_k z_k$ is a $(1/m^2_{2j+2},n_{2j+1})$-scc.  Using Proposition \ref{norm} the contradiction to our choice of $j$ is as follows,
$$\frac{\e}{m_{2j+1}} < \frac{1}{m_{2j+1}} \sum_{k =1}^d b_k g_k \bigg(\sum_{k=1}^d b_k T z_k \bigg) \leq \bigg\|T \bigg(\sum_{k=1}^d b_k z_k \bigg)\bigg\| \leq \frac{C\|T\|}{m^2_{2j+1}}.$$

\end{proof}

All the above yield the following:

\begin{prop}
\label{diag+ss} Let $Y$ be a block subspace of $\xfr_{wh}$ and
$T:Y\to Y$ a bounded linear operator. Then $T$ has the form
$T=D+S$ where $D:Y\to Y$ is a diagonal operator and $S:Y\to Y$ is
strictly singular.
\end{prop}
\begin{proof}

We set $D(y_n)=y^*_n(Ty_n)y_n$. Then clearly $D$ is diagonal and
bounded. By the previous Proposition we have that $T-D$ is a
strictly singular operator.

\end{proof}

As a consequence we obtain the following:

\begin{thm} Let $Y$ be a block subspace of $\xfr_{wh}$. Then $Y$
is not isomorphic to any of its proper subspaces.
\end{thm}

For a proof of the above we refer the interested reader to Corollary 30 in \cite{GM2}.

Similar arguments as the ones used in the preceding Theorem and Proposition \ref{diag+ss} can be applied to a more general setting. Namely, if $Y$ is a subspace generated by a sequence $(y_n)_{n=1}^{\infty}$ of disjointly supported vectors and $T\in \L (Y)$ then $T$ takes the form
$T=D+S$, where $D$ is diagonal and $S$ strictly singular.

\begin{rem}\label{restriction} Let $Y$ be a subspace of $\xfr_{wh}$
generated by a sequence $(y_n)_{n=1}^{\infty}$ with pairwise disjoint supports and $T\in \L (Y)$.
Then there exists a diagonal operator $D':\xfr_{wh}\to\xfr_{wh}$ and a strictly singular operator $S\in\L (Y)$
such that $T=D'|_{Y}+S$.
\end{rem}

\begin{proof} By Proposition \ref{diag+ss} there exist a diagonal operator $D\in\L (Y)$ and a strictly singular $S\in\L (Y)$ such that $T=D+S$. For each $n\in \N$ there exists a $\l_n\in\mathbb{R}$ such that $T(y_n)=\l_n y_n$. Let $i\in \N$. Set $D'e_i=\l_n e_i$ if $i\in \supp y_n$ and $D'e_i=0$ otherwise.
It can be readily seen that $D_Y$ is a diagonal operator in $\L (\xfr_{wh})$ and that the restriction of $D'$ on $Y$ coincides with $D$.
\end{proof}

This remark gives rise to the following problem.

\begin{prob} Let $Z$ be an arbitrary infinite dimensional closed
subspace of $\xfr_{wh}$. Let $T:Z\to Z$ be a bounded linear
operator. Do there exist $D:\xfr_{wh}\to\xfr_{wh}$ diagonal and
$S$ strictly singular such that $T=D|_{Z}+S$? Moreover, does there
exist an infinite dimensional closed subspace $Z$ of $\xfr_{wh}$
which is isomorphic to one of its proper subspaces?
\end{prob}

The arguments in the proof of Proposition \ref{diag+ss} yield the
following,

\begin{prop}\label{totinc} Let $Y,Z$ be two disjointly supported
block subspaces of $\xfr_{wh}$. Then every bounded linear operator
$T:Y\to Z$ is strictly singular.
\end{prop}
We recall that a Banach space $X$ is quasi-minimal if every two
infinite dimensional closed subspaces of $X$ are {\em not} totally
incomparable.
\begin{cor} \label{quasi} The space $\xfr_{wh}$ does not
contain a quasi minimal subspace.
\end{cor}
\begin{proof} Suppose $Z$ is a quasi minimal subspace of
$\xfr_{wh}$. Let $Y$ be a subspace of $Z$ isomorphic to a subspace
of $\xfr_{wh}$ generated by a block sequence $(y_n)_{n\in\N}$.
Proposition \ref{totinc} yields that the spaces generated by
$(y_{2n})_{n\in\N}$ and $(y_{2n+1})_{n\in\N}$ are totally
incomparable. This contradicts the assumption that $Y$ is quasi
minimal and the proof is complete.
\end{proof}

As it was mentioned in the introduction using the same method one
can construct a space $\xfr_{wh}^C$ over the field of complex
numbers that shares the same properties as $\xfr_{wh}$. At this
point we consider some of the spectral properties of a $T\in \L
(\xfr_{wh})$. For the rest of this section we abuse notation and
denote by $\xfr_{wh}$ both the real and the complex Banach space
discussed above. For every bounded linear operator $T$ that is
considered we let $\sigma (T)$ be its spectrum and $\sigma_p (T)$
its point spectrum. We start with the following result.

\begin{lem}\label{spdiag} Let $D\in \L (\xfr_{wh})$ be a diagonal operator with $De_n=\l_n e_n$, for $n\in \N$. Then $\sigma (D)=\overline{\{\l_n\}_{n=1}^{\infty}}$.
\end{lem}
\begin{proof} Let $\l\in\overline{\{\l_n\}_{n=1}^{\infty}}$. Choose a subsequence $(\l_{k_n})_{n=1}^{\infty}$ such that $\l_{k_n}\to\l$. Observe that $De_{k_n}-\l e_{k_n}\to 0$ and thus $\l \in \sigma (D)$. Now, suppose that $\l\notin\overline{\{\l_n\}_{n=1}^{\infty}}$. Then there exists $\epsilon>0$ with $|\l-\l_n|>\epsilon$. As the basis of $\xfr_{wh}$ is unconditional this yields that $D-\l I$ is invertible and the proof is complete.
\end{proof}

The following result correlates the spectrum of an arbitrary $T\in\L (\xfr_{wh})$ with that of its diagonal part.

\begin{prop}\label{spec} Let $T\in\L (\xfr_{wh})$ with $T=D+S$ where $D,S\in\L (\xfr_{wh})$ and $D$ is diagonal and $S$ strictly singular. Then the following hold,
\begin{enumerate}
\item $\sigma (T)\setminus (\sigma_p (T)\cup \{0\})\subset\sigma (D)$;
\item If $\l_n$ is an eigenvalue of $D$ with infinite dimensional eigenspace, then $\l_n\in \sigma(T)$;
\item  $\sigma (D)\setminus (\sigma_p(D)\cup\{0\})\subset \sigma (T)$.
\end{enumerate}
\end{prop}
\begin{proof}
\n (1) Let $\l\in \sigma (T)\setminus (\sigma_p (T)\cup \{0\})$
and suppose towards a contradiction that $\l\notin \sigma (D)$.
Then the operator $D-\l I$ is invertible and therefore a Fredholm
operator of index $0$. Standard Fredholm theory (see for example
Proposition 2.c.10 in \cite{LT}) that $D+S-\l I$ is also Fredholm
of index zero. By our assumptions, $T-\l I=D+S-\l I$ is not
invertible and as it is Fredholm $\l\in \sigma_p(T)$ which is a
contradiction and the proof is complete.

\n (2) By our assumptions there exists an infinite dimensional
subspace $Y$ of $\xfr_{wh}$ such that $(D-\l_n I)|_{Y}=0$.
Therefore, $(T-\l_n I)|_Y=S|_Y$ which immediately yields that
$T-\l_n I$ is not invertible.

\n (3) Let $\l\in\overline{\{\l_n\}_{n=1}^{\infty}}$ with
$\l\notin \{\l_n:n\in\N\}\cup\{0\}$. Choose a subsequence
$(\l_{k_n})_{n=1}^{\infty}$ such that $\l_{k_n}\to\l$. Since
$De_{k_n}-\l e_{k_n}\to 0$ and $S$ is strictly singular we can
find a subspace $Y$ of $<e_{k_n}:n\in\N>$ such that the operators
$D-\l I$ and $S$ are both compact on $Y$. Hence, $T-\l I$ is not
invertible.
\end{proof}


\section{The proof of Proposition \ref{7.8}}\label{AL}

The aim of this section is primarily to prove Proposition
\ref{7.8} which was stated in Section \ref{MP}. We start with the
following.

\begin{prop}\label{7.4}
Let $(x_k)_{k=1}^d$ be a $(C,2j_k)$ an exact sequence.  Let
$(f_\ell)_{\ell=1}^r \subset D_{wh}$ be $\ssh_{n_q}$-allowable
with $q<2j_1$.  Assume further that for each $k \in \{1, \ldots ,
d\}$, $\max\{\w(f_\ell): \supp x_k \cap \supp f_\ell
\not=\emptyset \} < m_{2j_k}$. Then for all $(\l_\ell)_{\ell
=1}^r, (b_k)_{k=1}^d \in Ba(\ell_2)$,

$$\sum_{\ell =1}^r \l_\ell f_\ell \bigg( \sum_{k=1}^d b_k x_k \bigg) \leq \frac{16C}{\displaystyle \min_{1 \leq \ell \leq r } \w(f_\ell)}.$$

\end{prop}

For the proof of Proposition \ref{7.4} we shall follow a similar
strategy to that of Proposition \ref{7.10}. Namely, we shall first
consider a partition of the vectors $(x_k)_{k=1}^d$ and then
proceed with the evaluation of $\sum_{\ell=1}^r \l_{\ell}f_{\ell}$
on each part separately.

For each $k \in \{ 1, \ldots , d \}$, let $x_k = m_{2j_k} \sum_{i=1}^{p_k} b_{k,i}x_{k,i}$ where $((x_{k,i})_{i=1}^{p_k})_{k=1}^d$ is a $(C,(2j_{k,i}))$ RIS sequence
in the lexicographical ordering.  Recall that $j_k < j_{k,i} < j_{k+1} $ for all $k \in \{ 1, \ldots, d\}$ and $i \in \{1,\ldots , p_k\}$. We will partition the support of each $x_{k,i}$.  For each $k \in \{1, \ldots , d\}$ and $i \in \{1, \ldots , p_k\}$ let,

\noindent $G_{k,i} =\{ t \in \A: ~t \text{ is maximal } \supp f_t \cap \supp x_{k,i}\not=\emptyset,
 \frac{1}{m(t)} \leq \frac{1}{m_{2j_{k,i+1}}}\}$.\\
Let $P_{k,i}$ denote the maximal $t \in \mathcal{A}$ such that $ \supp f_t \cap \supp x_{k,i} \not= \emptyset $ and $t \not\in G_{k,i}$. \\
For $t \in P_{k,i}$, let $\b_t=t$ if $w(f_{\g})<m_{2j_k}$ for all $\g\prec t$ or $\b_t=\min\{\g\prec t: w(f_{\g})\geq m_{2j_k}\}$ otherwise.

We set $S_{k,i}=\{\b_t: t\in P_{k,i},
\frac{1}{m(\b_t)}\leq\frac{1}{m_{2j_k}^3}\}$, $R_{k,i}=\{\b_t: t\in P_{k,i}, \frac{1}{m(\b_t)}>\frac{1}{m_{2j_k}^3}\}$.\\
We split $R_{k,i}$ into: $R_{k,i}^1=\{\b_t\in R_{k,i}:
w(f_{\b_t})\geq m_{2j_{k,i}}\}$, $R_{k,i}^2=\{\b_t\in R_{k,i}: w(f_{\b_t})< m_{2j_{k,i}}\}$.\\
Using the above sets we partition $x_{k,i}$ in the following way:
Set $x_{k,i}^G=x_{k,i}|_{\cup_{t\in G_{k,i}}\supp f_t}$,
 \\ $x_{k,i}^S=x_{k,i}|_{\cup_{\b_t\in S_{k,i}}\supp f_{\b_t}}$, $x_{k,i}^{R_1}=
 x_{k,i}|_{\cup_{\b_t\in R^1_{k,i}}\supp f_{\b_t}}$
and $x_{k,i}^{R_2}=x_{k,i}|_{\cup_{\b_t\in R^2_{k,i}}\supp
f_{\b_t}}$.

We also set $x_k^I=m_{2j_k}\sum_{i=1}^{p_k} b_{k,i} x^I_{k,i}$,
where $I=G,S,R^1,R^2$ and we denote by $G_k,S_k,R^1_k,R^2_k$ the
union for all $i$ of the corresponding
$G_{k,i},S_{k,i},R_{k,i}^1,R_{k,i}^2$.

We present the proof of Proposition \ref{7.4}
\begin{proof}

 {\bf (1)} We start by showing that, $\sum_{\ell =1}^r \l_\ell f_\ell (\sum_{k=k_0}^d  b_k
x_k^G) < \frac{2}{m_{2j_1}^2}$.We can readily see that
$\|f_{\ell}|_{\supp
 x_{k,i}^G}\|_{\infty}\leq\frac{1}{m_{2j_{k,i+1}}}$ for all
 $\ell=1,...,r$, $k=1,...,d$ and $i=1,...,p_k$. Thus applying the
 RIS estimate (RISE) on $x_k^G$ we can see that
 $\sum_{\ell=1}^r
 \l_{\ell}f_{\ell}(x_k^G)\leq\frac{1}{m_{2j_k}^2}$. Summing up for
 all $k$ yields the result.

{\bf (2)} We pass to the following evaluation :
\[\sum_{\ell =1}^r\l_\ell f_\ell
\bigg(\sum_{k=1}^d  b_k x_k^S\bigg)< \frac{10C}{m_{2j_1}}\]
 For each $\b_t \in S_k $ set $a_t=\min\{a \prec
\b_t:\frac{1}{m(a)}<\frac{1}{m_{2j_k}^2}\}$. Observe that
$\{a_t:\b_t\in S_k\}$ is a maximal antichain for the tree
representation of the functional $\sum_{\ell =1}^r \l_\ell
f_\ell|_{\supp x_k^G}$. We can also see that
$\frac{1}{m(a_t)}\geq\frac{1}{m_{2j_k}^3}$. Hence by Lemma
\ref{lem6} the family $\{f_{a_t}:\b_t\in S_k\}$ is $S_{n_{2jk}-1}$
allowable.

\[\sum_{\ell=1}^r \l_\ell f_\ell  (x_k^{S} ) \leq \frac{1}{ m_{2j_k}}  \sum_{\{a_t :~ \b_t \in S_k\}}
\frac{\l(a_t)}{m(a_t)} f_{a_t} \bigg(  \sum_{i=1}^{p_k} b_{k,i}
x_{k,i}^{S} \bigg)\]

A direct application of Lemma \ref{7.2} and a summation over all $k$ yields the estimate.  \\

{\bf (3)} At this point we prove that,
\[\sum_{\ell =1}^r \l_\ell
f_\ell \bigg(\sum_{k=1}^d b_k x_k^{R_1}\bigg)<
\frac{2C}{m_{2j_1}}.\]
 Fix $k \in \{1, \ldots, d\}$. For $\b\in
R_k^1$ let $h_\b = f_\b|_{\supp x_k^{R_1}}$. Observe that
\[\sum_{\ell=1}^r\l_\ell f_\ell  (x_k^{R_1}
)=m_{2j_k}\sum_{\b\in
R_k^1}\frac{\l(\b)}{m(\b)}h_{\b}\bigg(\sum_{i=1}^{p_k}b_{k,i}x^{R_1}_{k,i}\bigg)\]

By Remark \ref{rem3} and the fact that for $k\neq k'$ it holds
that$R^1_k\cap R^1_{k'}=\emptyset$. The family $\{h_{\b}:\b\in
R_k^1\}$ is $S_{n_{2j_k}-1}$ allowable and $\{\maxsupp
x^{R_1}_{k,i}:i\leq p_k, x_{k,i}^{R_1}\neq 0\}\in S_{n_{2j_k}-1}$
allowable. Therefore, $\frac{1}{m_{2j_k}}\sum_{b\in
R_k^1}\frac{\l(\b)}{m(\b)}h_{\b}=g_k\in D_{wh}$ and
$(\sum_{\{i:x_{k,i}^{R_1}\neq
0\}}b_{k,i}^2)^{1/2}<\frac{1}{m_{2j_k}^3}$.\\
Thus applying the upper $\ell_2$ estimate,

$\sum_{\ell=1}^r\l_\ell f_\ell (x_k^{R_1}
)=m^2_{2j_k}g_k(\sum_{\{i:x_{k,i}^{R_1}\neq
0\}}b_{k,i}x^{R_1}_{k,i})\leq\frac{C}{m_{2j_k}}$.

A summation over all $k$ yields the estimate.

{\bf (4)} The last estimate is, \[\sum_{\ell =1}^r \l_\ell f_\ell
\bigg(\sum_{k=1}^d b_k x_k^{R_2}\bigg)< \frac{12C}{\displaystyle
\min_{1 \leq \ell \leq r} \w(f_\ell)}\]
 For $k \in \{1, \ldots,
d\}$ and $\b\in R_k^2$ let $h_\b = f_\b|_{\supp x_k^{R_2}}$. It is
clear that for $\b\neq \b'$, $\supp f_{\b}\cap \supp
f_{\b'}=\emptyset$. For $\b\in R_k^2$ not maximal let $S_{\b}$
denote its immediate successors in $\mathcal{A}$. For $\g\in
S_{\b}$ set $h_{\g}=f_{\g}|_{\supp x_k^{R_2}}$. For given
$k\in\{1,...,d\}$ and $i\in\{1,...,p_k\}$ write
$x^{R_2}_{k,i}=\sum_{j=1}^{p_{k,i}}b_{k,i,j}x_{k,i,j}^{R_2}$. We
define the following sets:
\begin{equation*}
\begin{split}
\Phi^C= \{ (k,i,j):&~  \exists \b\in R_k^2 \text { not maximal},~\g \in S_{\b} \text{ with }  \minsupp h_\g \in \ran x^{R_2}_{k,i,j}\}\cup\\
& \{(k,i,j):\exists \b\in R_k^2\text{ maximal s.t. } \supp f_\b
\cap \supp x^R_{k,i,j} \not=\emptyset \ \}
\end{split}
\end{equation*}

$\Phi_{k,i}^C=\{j:(k,i,j)\in \Phi^C\}$, $\Phi=\{(k,i,j):k\leq
d,~i\leq p_k,~j\leq p_{k,i}\}\setminus\Phi$,
$\Phi_{k,i}=\{1,...,p_{k,i}\}\setminus \Phi_{k,i}^C$.

Since, $w(f_{\b})<m_{2j_{k,i}}$, for all $\b\in R_{k,i}^2$ it
follows that the family $\{h_{\g}:\g\in S_{\b}\}$ is a
$S_{n_{(2j_{k,i}-1)}}$ allowable family. In addition, the family
$\{h_{\b}:\b\in R_k^2\}$ is $S_{n_{(2j_k-1)}}$ allowable, hence
the family $\{h_{\g}:\g\in S_{\b},\b\in R_k^2\}\cup\{h_{\b}:\b
\text{ is maximal and } \supp h_{\b}\cap\supp
x_k^{R_2}\neq\emptyset\}$ is $S_{(n_{2j_k}-1)}$ allowable. Hence,
applying Lemma \ref{sae} (SAE),

\[( \sum_{\b \in R^2_k,~ \g
\in S_\b} \frac{\l(\g)}{m(\g)} h_\g +\sum_{\substack{\b \in \R^2_k\\
\b \text{ maximal}}}\frac{\l(\b)}{m(\b)} f_\b)(\sum_{j\in
\Phi^C_{k,i}}b_{k,i,j}x_{k,i,j}^{R_2})\leq\frac{C}{m_{2j_{k,i}}}.\]

Summing up for all $i\in\{1,..,p_k\}$ and for all
$k\in\{1,...,d\}$ yields the desired estimate.

The final estimate for this proposition concerns
$x_{k,i}^{R_2}$ such that $j \in \Phi_{k,i}$.

Before proceeding we need the following notion:\\
\noindent We call $\b \in R^2$ $s$-minimal if $|\{ \b' : \b' \prec
\b,~ \b' \in R^2\}| = s$. For $s\geq 0$ let $L_s =\{ \b \in \R^2 :
\b \text{ is } s-minimal \}$.

Observe that for $s\geq 0$ and $\b\neq \b'$ in $L_s$ the nodes
$\b,\b'$ are incomparable. Indeed, if we assume that $\b\prec
\b'$, then $|\{\g\in R^2: \g\prec\b'\}|=s\geq |\{\g\in R^2:\g\prec
\b\}|+1=s+1$ which is a contradiction showing that the nodes are
incomparable.

\begin{sub}\label{sub1}
For each $\g \in \cup S_{\b}$ let, $\eta_\g = \l(\g) \prod_{\ell
\prec \delta \prec \b} \frac{1}{\w(f_\delta)}$. Then, $\sum_{\g
\in \cup S_{\b}} \eta_\g^2 \leq 2$.
 \end{sub}
 \begin{proof}
By the definition of the set $R^2$ we deduce that for $\b\in L_s$
with $s>0$, $\prod_{\ell \prec \delta \prec \b}
\frac{1}{\w(f_\delta)} < \frac{1}{m_{2j_s}}$. As we saw earlier
for each $s>0$, the nodes in $L_s$ are incomparable.  Therefore,

\begin{equation*}
\begin{split}
\sum_{\g \in \cup S_{\b}} \eta_\g^2& =\sum_{\ell =1}^r \l_\ell^2  \sum_{\b \in R^2 \cap \mathcal{A}_\ell} \bigg(\prod_{\ell \prec \delta \prec \b} \frac{\l_\delta}{\w(f_\delta)} \bigg)^2 \sum_{\g \in S_\b} \l_\g^2 \\
&= \sum_{s>0} \sum_{\ell =1}^r \l_\ell^2 \sum_{\b \in L_s\cap R^2 \cap \mathcal{A}_\ell} \bigg(\prod_{\ell \prec \delta \prec \b} \frac{\l_\delta}{\w(f_\delta)} \bigg)^2  \sum_{\g \in S_\b} \l_\g^2 \\
&\hspace{1in} +\sum_{\ell =1}^r \l_\ell^2\sum_{\b \in L_0\cap R^2 \cap \mathcal{A}_\ell} \bigg(\prod_{\ell \prec \delta \prec \b} \frac{\l_\delta}{\w(f_\delta)} \bigg)^2 \sum_{\g \in S_\b} \l_\g^2 \\
&\leq \bigg( \sum_{s>0} \frac{1}{m_{2j_s}}\bigg) +1 < 2.
\end{split}
\end{equation*}
\end{proof}

 Now, let

$$y= \bigg(\sum_{k=1}^d b_k m_{2j_k} \sum_{i=1}^{p_k} b_{k,i} \sum_{j \in \Phi_{k,i}} b_{k,i,j} x^{R_2}_{k,i,j} \bigg) $$

\n and

$$\tilde{y}= \bigg(\sum_{k=1}^d b_k \sum_{i=1}^{p_k} b_{k,i} \sum_{j \in \Phi_{k,i}} b_{k,i,j} x^{R_2}_{k,i,j} \bigg) $$

\n Observe that,

$$\sum_{\ell=1}^r \l_\ell f_\ell (y) = \sum_{k=1}^d\sum_{\ell =1}^r \frac{\l_\ell}{\w(f_\ell)} \sum_{\b \in R^2_k\cap \mathcal{A}_\ell} \prod_{\ell \prec \delta \prec \b} \frac{\l_\delta}{\w(f_\delta)} \frac{1}{\w(f_\b)} \sum_{\g \in S_\b} \l_\g h_\g (y).$$

\n For each $\b \in R^2_k$ use the fact that $\w(f_\b) \geq
m_{2j_k}$ to continue as follows,

$$ \sum_{\ell=1}^r \l_\ell f_\ell (y)  \leq \frac{1}{\displaystyle \min_{1 \leq \ell \leq r} \w(f_\ell)} \sum_{k=1}^d\frac{1}{m_{2j_k} } \sum_{\ell =1}^r \l_\ell  \sum_{\b \in R^2_k \cap \mathcal{A}_\ell } \prod_{\ell \prec \delta \prec \b} \frac{\l_\delta}{\w(f_\delta)} \sum_{\g \in S_\b} \l_\g  h_\g(y) .$$
Since $\cup_{\b \in R^2_k} \supp h_\b = \supp x_k^{R_2}$ each
$m_{2j_k}$ term cancels to yield,

$$ \sum_{\ell=1}^r \l_\ell f_\ell(y) \leq \frac{1}{\displaystyle \min_{1 \leq \ell \leq r} \w(f_\ell)}  \sum_{\g \in \cup S_{\b}} \eta_\g h_\g (\tilde y).$$

\n The final estimate follows by applying Lemma \ref{mfe} and
Sublemma \ref{sub1}. Combining the estimates {\bf (1),(2),(3)},and
{\bf (4)} yields the result.
\end{proof}

We pass now to show the main result of this section, namely the
proof of Proposition \ref{7.8}.
\begin{proof}{\bf (Proof of Proposition \ref{7.8})}.
 For each $k \in \{1, \ldots, d\}$, $x_k$ is of the form, $x_k= m_{2j_k} \sum_{i=1}^{p_k} b_{k,i} x_{k,i}.$

Note that $m_{2j_k}$ $b_{k,i} < 1/m_{2j_k}$ hence $\|x_k \|_\infty <1$. We shall decompose $\{1,\ldots,r\}$  as follows;

       $$E_0 =\{\ell: w(f_\ell) < m_{2j_1}\}$$,

For $k\in \{1,...,d\}$

       $$E_{k,0} =\{\ell: m_{2j_k} \leq w(f_\ell)< m_{2j_{k,1}} \}$$

\n For $k\in \{1,...,d\}$  and  $1\leq i < p_k$;

 $$E_{k,i} =\{\ell: m_{2j_{k,i}} \leq w(f_\ell)< m_{2j_{k,i+1}} \}$$

\n and for  $1\leq k < d$,

      $$ E_{k,p_k} = \{\ell: m_{2j_{k,p_k}} \leq w(f_\ell) < m_{2j_{k+1}} \}$$

\n and finally,

       $$E_{d,p_d} = \{ l: m_{2j_{d,p_d}} \leq w(f_\ell) \}.$$

Next, according to the set that some $\ell$ belongs to, we split
the functional $f_\ell$ into at most four parts denoted as: $
f_\ell^\varepsilon$  with $\varepsilon = 1,2,-1,-2$ as follows.
Let $\ell\in E_{k,i}$ then, $f_\ell^{-2} = f_\ell |_{[1, \maxsupp
x_{k,i-1}]}$ if $1<i$ and $f_\ell^{-2} = f_\ell |_{[1, \maxsupp
x_{k-1}]}$,if $i=1$.\\
\noindent $f_\ell^{-1} = f_\ell |_{[\supp x_{k,i}]}$,\\
$f_\ell^{1} = f_\ell|_{[\minsupp x_{k,i+1},\maxsupp x_{k}]}$, if
$i<p_k$ and $f_{\ell}^{1}=0$,if $i=p_k$.\\
\noindent $f_\ell^{2} = f_\ell|_{[\minsupp x_{k+1}, \infty)}$ if
$k<d$ and $f_{\ell}^{2}=0$,if
$k=d$.\\

 For $\ell\in E_{k,0}$ set $ f_\ell^{-2} = f_\ell |_{[1,
\maxsupp x_{k-1}]}$, $f_\ell^{-1} = f_\ell |_{[\supp x_{k}]}$,
$f_{\ell}^1=0$ and $f_\ell^{2} = f_\ell |_{[\minsupp x_{k+1},
\infty)}$. Finally for $\ell\in E_0$ set
$f_{\ell}^{-2}=f_{\ell}^{-1}=f_{\ell}^{1}=0$ and
$f_{\ell}^{2}=f_{\ell}$.

We proceed to the estimates:

Notice that for $\varepsilon\in\{-2,...,2\}$

\begin{equation*}
\begin{split}
\sum_{\ell=1}^r \l_{\ell}f_{\ell}^{\varepsilon}(\sum_{k=1}^d b_k
x_k) &= \sum_{\ell\in
E_0}\l_{\ell}f_{\ell}^{\varepsilon}(\sum_{k=1}^d b_k
x_k)+\sum_{t=1}^d\sum_{\ell\in E_{t,0}}
\l_{\ell}f_{\ell}^{\varepsilon}(\sum_{k=1}^d b_k x_k)+\\
& \sum_{t=1}^d\sum_{j=1}^{p_k}\sum_{\ell\in E_{t,j}}
\l_{\ell}f_{\ell}^{\varepsilon}(\sum_{k=1}^d b_k x_k)
\end{split}
\end{equation*}
{\bf (1)} For $\ell\in E_0$ Proposition \ref{7.4} yields
$\sum_{\ell\in E_0}\l_{\ell}f_{\ell}^{2}(\sum_{k=1}^d b_k x_k)\leq
\frac{16C}{\min_{\ell} w(f_{\ell})}$, while for $\varepsilon\neq
2$ the corresponding sum is equal to zero.

{\bf (2)} For the sum \[\sum_{t=1}^d\sum_{\ell\in E_{t,0}}
\l_{\ell}f_{\ell}^{-2}(\sum_{k=1}^d b_k x_k)+
\sum_{t=1}^d\sum_{j=1}^{p_k}\sum_{\ell\in E_{t,j}}
\l_{\ell}f_{\ell}^{-2}(\sum_{k=1}^d b_k x_k),\] we can directly
apply the RISE (Lemma \ref{rise}) and obtain
\[\sum_{\ell=1}^r \l_{\ell}f_{\ell}^{-2}(\sum_{k=1}^d b_k
x_k)\leq\frac{2}{m_{2j_1}}.\]

{\bf (3)} Note that,

\[\sum_{t=1}^d\sum_{\ell\in E_{t,0}}
\l_{\ell}f_{\ell}^{-1}(\sum_{k=1}^d b_k x_k)=\sum_{k=1}^d
m_{2j_k}b_k\sum_{\ell\in
E_{k,0}}\l_{\ell}f_{\ell}^{-1}(\sum_{i=1}^{p_k} b_{k,i}
x_{k,i}).\]

Observe that for all $\ell\in E_{k,0}$, $w(f_{\ell})>m_{2j_k}$ if
$f_{\ell}^{-1}\neq 0$ by our assumption. In addition, the family
$(f_{\ell}^{-1})_{\ell\in E_{k,0}}$ is $S_{n_q}$ allowable and
$w(f_{\ell}^{-1})<m_{2j_{k,1}}$. Hence, Proposition \ref{7.4}
yields that,

\[\sum_{\ell\in
E_{k,0}}\l_{\ell}f_{\ell}^{-1}(\sum_{i=1}^{p_k} b_{k,i}
x_{k,i})\leq \frac{12C}{\min_{\ell} w(f_{\ell})}.\]

Therefore, the estimate is a result of a summation for all $k$.

For the term $\sum_{t=1}^d\sum_{j=1}^{p_k}\sum_{\ell\in E_{t,j}}
\l_{\ell}f_{\ell}^{-1}(\sum_{k=1}^d b_k x_k)$ we may apply the SAE
(Lemma \ref{sae}) to see that,

\[\sum_{j=1}^{p_k}\sum_{\ell\in E_{t,j}}
\l_{\ell}f_{\ell}^{-1}(m_{2j_k}\sum_{i=1}^{p_k} b_{k,i}
x_{k,i})\leq \frac{C}{m_{2j_k}}.\]

Again summing up for all $k$ yields the desired estimate.

{\bf (4)} Notice that $\sum_{\ell\in E_{k,i}}
\l_{\ell}f_{\ell}^{1}(x_k)=\sum_{\ell\in E_{k,i}}
\l_{\ell}f_{\ell}^{1}(m_{2j_k}\sum_{t=i+1}^{p_k} b_{k,t}
x_{k,t})$. Consequently applying Lemma \ref{7.3} we have,

\[\sum_{\ell\in E_{k,i}}
\l_{\ell}f_{\ell}^{1}(x_k)\leq\frac{5Cm_{2j_k}}{m_{2j_{k,i}}}\]
Thus, summing up for all $i$ yields
\[\sum_{i=1}^{p_k}\sum_{\ell\in E_{k,i}}
\l_{\ell}f_{\ell}^{1}(x_k)\leq\sum_{i=1}^{p_k}\frac{5C
m_{2j_k}}{m_{2j_{k,i}}}\leq\frac{10C}{m_{2j_k}}.\]

Hence, the summation for all $k$ gives the desired evaluation.

{\bf (5)} Finally, for the sum
$\sum_{t=1}^d\sum_{j=1}^{p_k}\sum_{\ell\in E_{t,j}}
\l_{\ell}f_{\ell}^{2}(\sum_{k=1}^d b_k x_k)$ we can verify that
the conditions of Proposition \ref{7.4} are fulfilled and a direct
application yields the result.
\end{proof}

\def\cprime{$'$} \def\cprime{$'$}

\end{document}